\documentclass{siamltex}		% Specifies the document style. 
\usepackage{graphicx,multirow,comment,amssymb,amsmath,amsfonts,subfigure,
hyperref,color}
\usepackage[normalem]{ulem}

\def\HH{\mathcal{H}}

\def\RR{\mathbb{R}}

\def\tF{\texttt{F}}
\def\invt{^{-T}}%
\def\inv{^{-1}}%
\def\nref#1{(\ref{#1})}
\def\tS{\tilde{S}}

\newcommand{\xpmatrix}[1]{\begin{pmatrix} #1 \end{pmatrix}}
 
\newcommand{\eq}[1]{\begin{equation}\label{#1}}
\newcommand{\en}{\end{equation}}

\def\eps{\varepsilon} 
\graphicspath{{./FIGS/}}

\title{Schur Complement based 
domain decomposition preconditioners with Low-rank corrections
\thanks{This work was supported
by NSF under grant NSF/DMS-1216366 and by the Minnesota Supercomputing
Institute}}

\author{Ruipeng Li \and Yuanzhe Xi \and Yousef Saad
\thanks{Address: Department of Computer Science \& Engineering,
        University of Minnesota, Twin Cities.
{\tt \{rli,yxi,saad\} @cs.umn.edu}}
}

\begin{document}

\maketitle

\begin{abstract}
This  paper  introduces a  robust  preconditioner  for general  sparse
symmetric matrices,  that is based  on low-rank approximations  of the
Schur complement in a Domain Decomposition (DD) framework.  In this ``Schur
Low Rank''  (SLR) preconditioning approach, the  coefficient matrix is
first  decoupled   by  DD,  and   then  a  low-rank
correction is exploited to compute an approximate inverse of the Schur
complement associated  with the  interface points.  The  method avoids
explicit  formation  of  the  Schur  complement matrix.  We  show  the
feasibility  of this  strategy  for  a model  problem,  and conduct  a
detailed spectral  analysis for the relationship  between the low-rank
correction and the quality  of the preconditioning.  Numerical
experiments  on general  matrices illustrate the
robustness and efficiency of the proposed  approach.
\end{abstract}


\section{Introduction}
We consider the problem of solving the linear system
\begin{equation} \label{eq:linsys}
Ax=b,
\end{equation}
with $A \in \RR^{n\times n}$ a large sparse \emph{symmetric} matrix.
Krylov subspace methods preconditioned with a form of Incomplete LU (ILU) 
factorization can be quite effective for this problem 
but there are situations where ILU-type
preconditioners encounter difficulties. 
For instance, when the matrix is highly
ill-conditioned or indefinite, the construction of the factors 
may not complete or may result in unstable factors.
Another situation, one that initially motivated this line of work,
is that ILU-based methods can yield exceedingly poor performance
on certain high performance computers such as those 
equipped with  GPUs \cite{RliSaadGPU} or Intel Xeon Phi processors.
This is because 
building and using  ILU factorizations is a highly sequential process. 
Blocking, which is a highly effective strategy utilized by 
sparse direct solvers to boost performance, is rarely exploited in 
the realm of iterative solution techniques. 

In the late 1990s, a class of methods appeared as an alternative
to ILUs, that did not require forward and backward triangular solves.
These were developed  primarily as a means to bypass the issues just
mentioned and  were based on finding an approximate inverse of the original
matrix, that was also sparse, see, e.g., \cite{Benzi-al,BenT98,grote-huckle}
among others. These methods were, by and large, 
later abandoned as practitioners 
found them too costly  in terms of preprocessing, iteration time,
and memory usage. 
\begin{comment} 
These  situations have  motivated the  development of  the approximate
inverse  preconditioners  as   alternatives.   However,  the  cost  of
developing these preconditioners and the number of iterations required
for convergence are often high,  and generally these approaches do not
work  very  well  for  highly  indefinite matrices,  as  a  well-known
weakness.
\end{comment}

Another line of work that emerged in recent years 
as a means to compute preconditioners, is that of
\emph{rank-structured matrices}.
The starting point is the work  by W. Hackbusch and co-workers
who introduced the notion of $\HH$-matrices in the 1990s 
\cite{Hackbusch99Hmat1,Hackbusch99Hmat2}. These were based on some interesting
rank-structure observed on matrices arising from the use of the fast multipole
methods or the inverses of some partial differential operators.
A similar  rank-structure was also exploited by others in 
the so-called  Hierarchically Semi-Separable (HSS) matrix format which
represents  certain  off-diagonal  blocks by  low-rank  matrices
\cite{DarveHSS-2O13,EngquistYing,LeBorne:2003,LeBorneGras06,WangdeHoopXia11,XiaChanGuLu10,XiaGu10}. 

More  recent  approaches did not exploit this rank structure but  
focused instead   on  a  multilevel  low-rank  correction
technique,  which  include the  recursive  Multilevel Low-Rank  (MLR)
preconditioner  \cite{MLR-1},   Domain  Decomposition  based  Low-Rank
(DD-LR)    preconditioner   \cite{RliSaadMLR2},    and    the   LORASC
preconditioner  \cite{grigori:hal-01017448}.  This  paper  generalizes
the technique developed in \cite{MLR-1} 
to the  classical Schur complement methods and proposes
a    Schur     complement    based    Low-Rank     (SLR)    correction
preconditioner.

This  paper  considers  only  symmetric  matrices,  and  the  proposed
spectral  analysis is  restricted to  the Symmetric  Positive Definite
(SPD)  case.  However,  the   method  can  be  extended  to  symmetric
indefinite matrices as  long as they have an  SPD interface, i.e., the
submatrix associated  with the  interface unknowns resulting  from the
partitioning is  SPD. This assumption  usually holds for  matrices
arising from  discretization of PDEs.
%%%[\textcolor{red}{citations  needed}].   
Extensions  to  the  symmetric  indefinite  matrices  with  indefinite
interface  matrices, as  well  as to  nonsymmetric  matrices are  also
possible but these will be explored in our future work.

It is  useful to compare the  advantages and the  disadvantages of the
proposed  approach with  those  of the  traditional  ILU-type and  the
approximate    inverse-type   preconditioners.     First,    the   SLR
preconditioner  is directly  applicable  to the  class of  distributed
linear systems  that arise in  any standard Domain  Decomposition (DD)
approach,  including all  vertex-based,  edge-based, or  element-based
partitionings.      Second,     it      is     well     suited     for
single-instruction-multiple-data (SIMD)  parallel machines.  Thus, one
can expect to implement this preconditioner on a multiprocessor system
based  on a  multi(many)-core  architecture exploiting  two levels  of
parallelism.   Third, as  indicted by  the experimental  results, this
method is  not as sensitive to  indefiniteness as ILUs or 
sparse approximate inverse preconditioners.   A fourth appeal, shared by
all the approximate inverse-type  methods, is that an SLR preconditioner
can be easily updated in  the sense that if it 
does not yield satisfactory performance, it can easily be improved without
forfeiting the work performed so far in  building it. 

The  paper is  organized  as follows:  in Section~\ref{sec:backg},  we
  introduce   the  DD   framework  and    Schur  complement
techniques.     A  spectral   analysis    will    be   proposed
Section~\ref{sec:specanaly}.  The SLR preconditioner will be discussed
in Section~\ref{sec:LowRank} followed by  implementation details
in Section~\ref{sec:impl}.  Numerical  results   of  model  problems  and
general   symmetric   linear   systems   are  presented   in   Section
\ref{sec:exp}, and we conclude in Section~\ref{sec:concl}.

\section{Background: sparse linear systems and the DD framework} 
\label{sec:backg}
In \cite{MLR-1}   we   introduced    a   method   based   on   a
divide-and-conquer  approach  that   consisted  in  approximating  the
inverse of a matrix $A$ by essentially the inverse of its $2 \times 2$
block-diagonal   approximation  plus   a  low-rank   correction.  This
principle  was  then  applied  recursively  to each  of  the  diagonal
blocks. We observed that there is often \emph{a
  decay  property} when  approximating the  inverse of  a matrix  by the
inverse  of a close-by  matrix in other contexts.
 By  this we  mean that  the difference
between the  two inverses has very rapidly  decaying eigenvalues, which
makes  it  possible  to  approximate  this  difference  by  small-rank
matrices. The best framework  where this property
takes  place is  that of  DD which is 
emphasized in this paper.

\subsection{Graph partitioning}
Figure \ref{fig:part} shows two 
standard ways of partitioning a graph \cite{Saad-book2}.
On the left side is a \emph{vertex-based} partitioning that is
common in the sparse matrix  community where it is also
referred to as graph partitioning by \emph{edge-separators}. A vertex
is an  equation-unknown pair and the partitioner subdivides the vertex
set into $p$ partitions, i.e., $p$ non-overlapping subsets
whose union is equal to the original vertex set.
On the right side is an \emph{edge-based} partitioning, which, in contrast,
consists of  assigning edges to subdomains.
This is also called graph partitioning by \emph{vertex separators} in the 
graph theory community.

\begin{figure}[h!]
\caption{Two classical ways of partitioning a graph, vertex-based partitioning
(left) and edge-based partitioning (right).
\label{fig:part}} 
\centerline{ 
\includegraphics[width=0.4\textwidth]{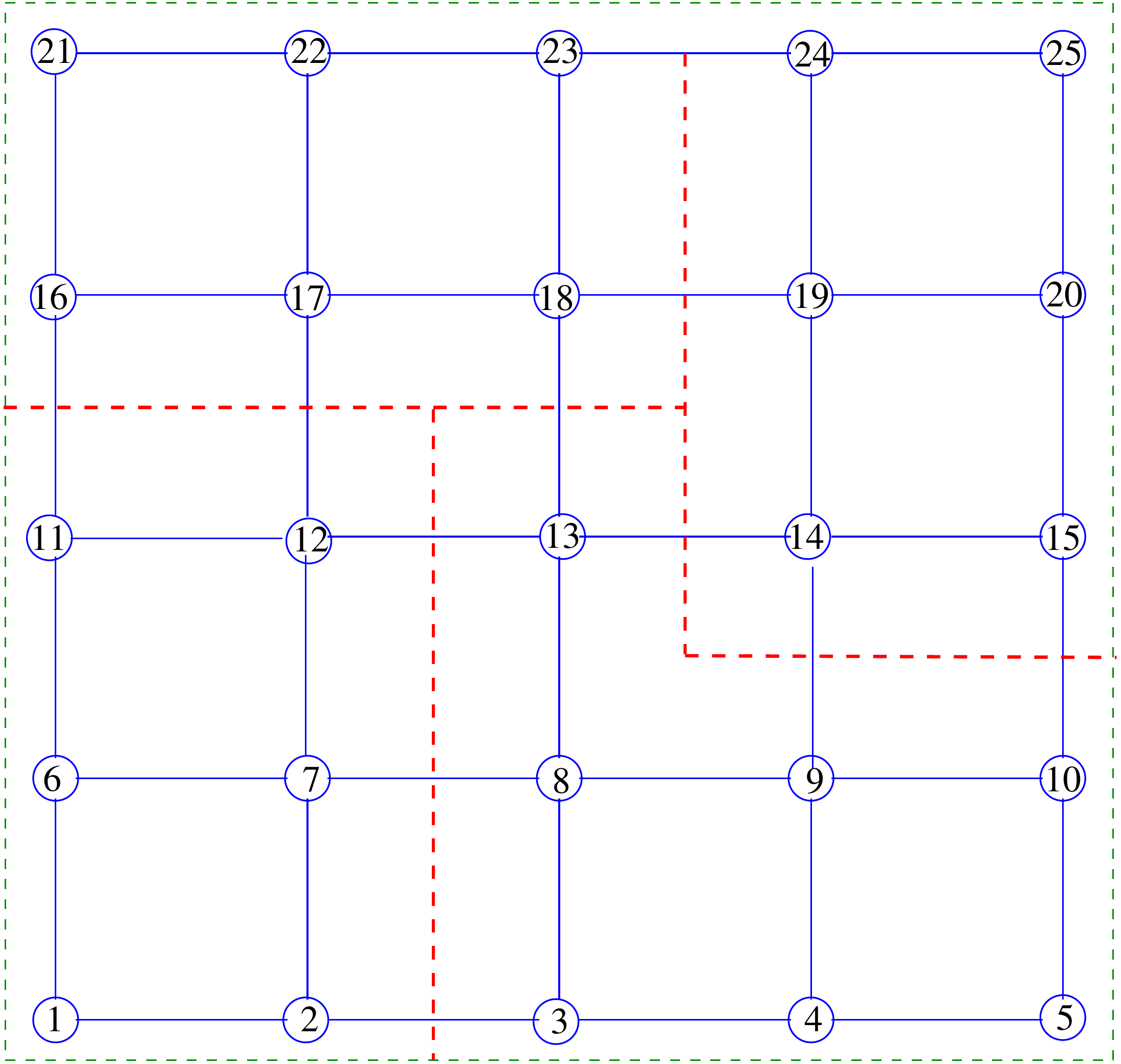}
\hspace{0.05\textwidth}
\includegraphics[width=0.4\textwidth]{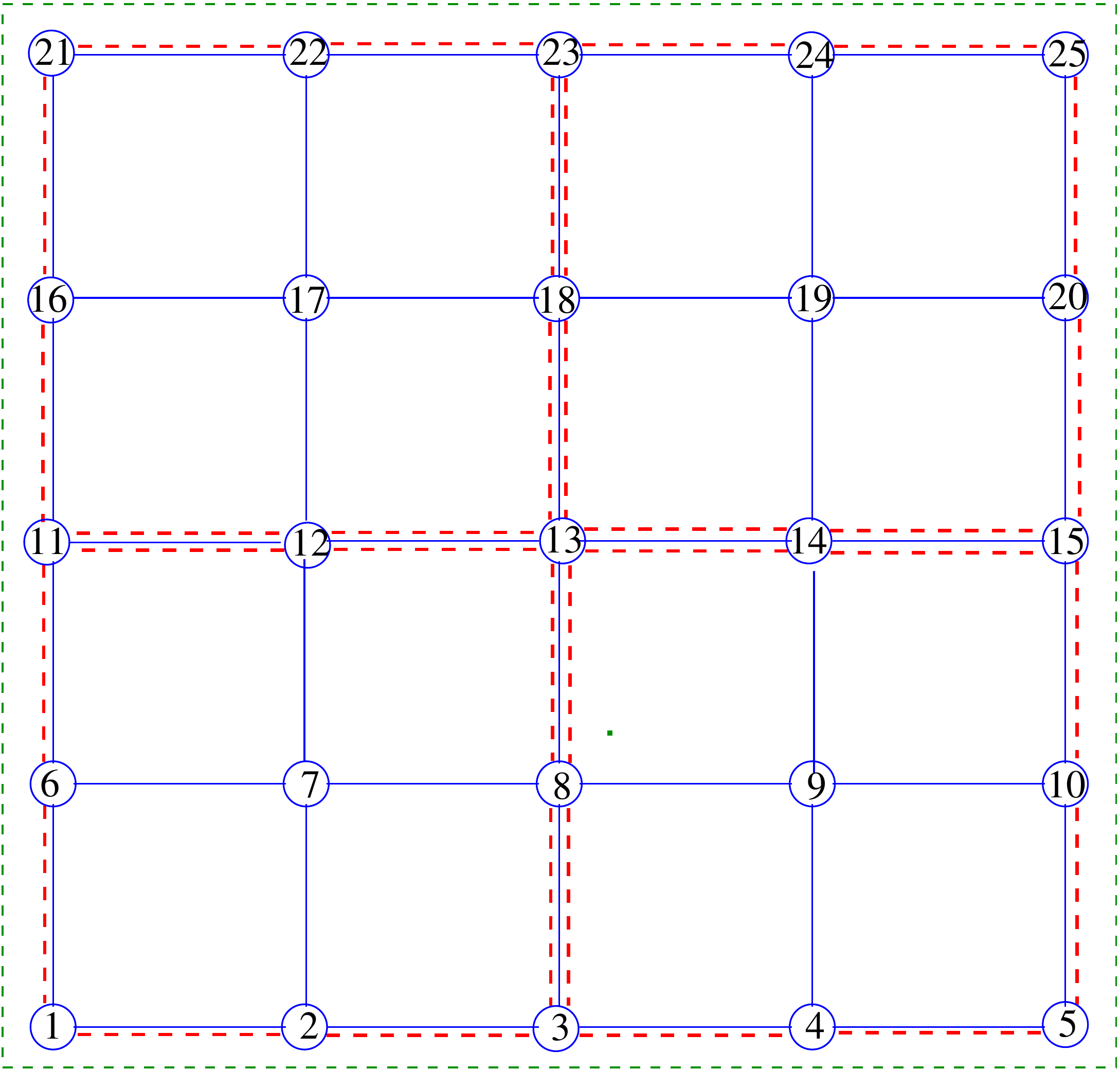}}
\end{figure} 

From the perspective of a  subdomain, one can distinguish $3$ types of
unknowns: (1) interior unknowns, (2)  local interface unknowns,
(3)  and   external   interface  unknowns.   
This  is   illustrated  on  the  top  of
Figure~\ref{fig:locsys}.
In   a  vertex-based partitioning, interior unknowns
are  those   coupled  only  with local  unknowns;
 local  interface
unknowns are those coupled with  both external and local unknowns; and
external  interface  unknowns  are  those   that  belong  to  other
subdomains and  are coupled with local interface unknowns.  
In an edge-based partitioning the local and external interface 
are merged into one set consisting all nodes that are \emph{shared} by 
a subdomain and its neighbors while interior nodes are those nodes
that are \emph{not shared}.

\begin{figure}[ht]
\caption{A local view of a distributed sparse matrix: vertex-based partitioning (top-left), edge-based partitioning (top-right), and the matrix representation (bottom). \label{fig:locsys}}
\begin{center}
\vspace{0.5em}
\subfigure{
\includegraphics[scale=0.4]{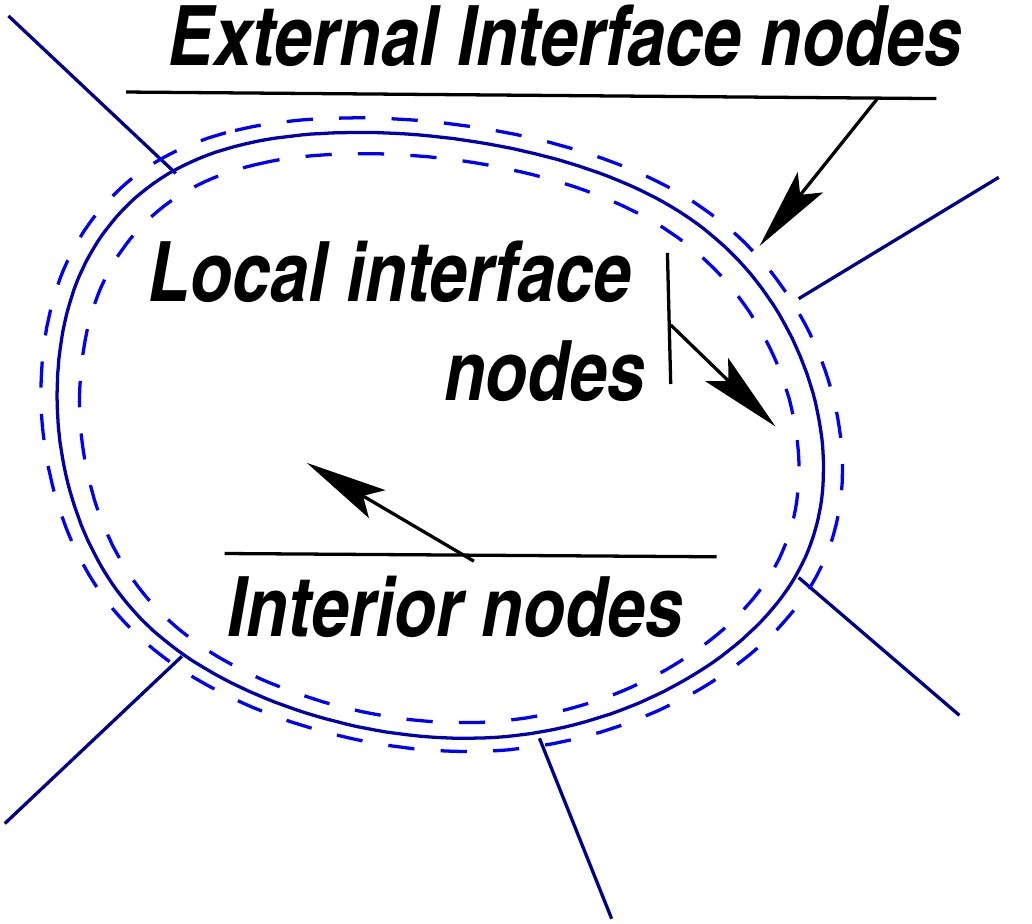}
}
\hspace{2em}
\vspace{0.5em}
\subfigure{
\includegraphics[scale=0.4]{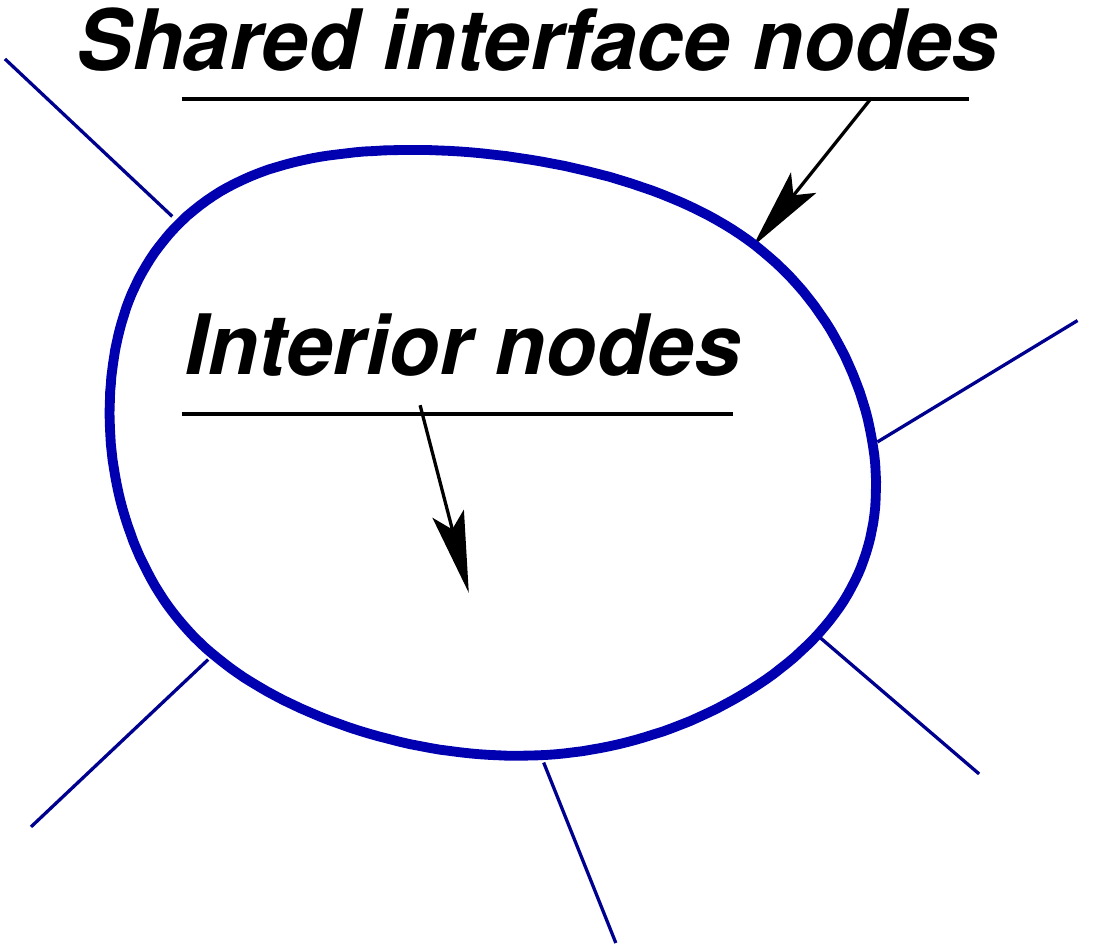}
}
%\hspace{0.5em}
%\vspace{-1em}
\subfigure{
\includegraphics[scale=0.4]{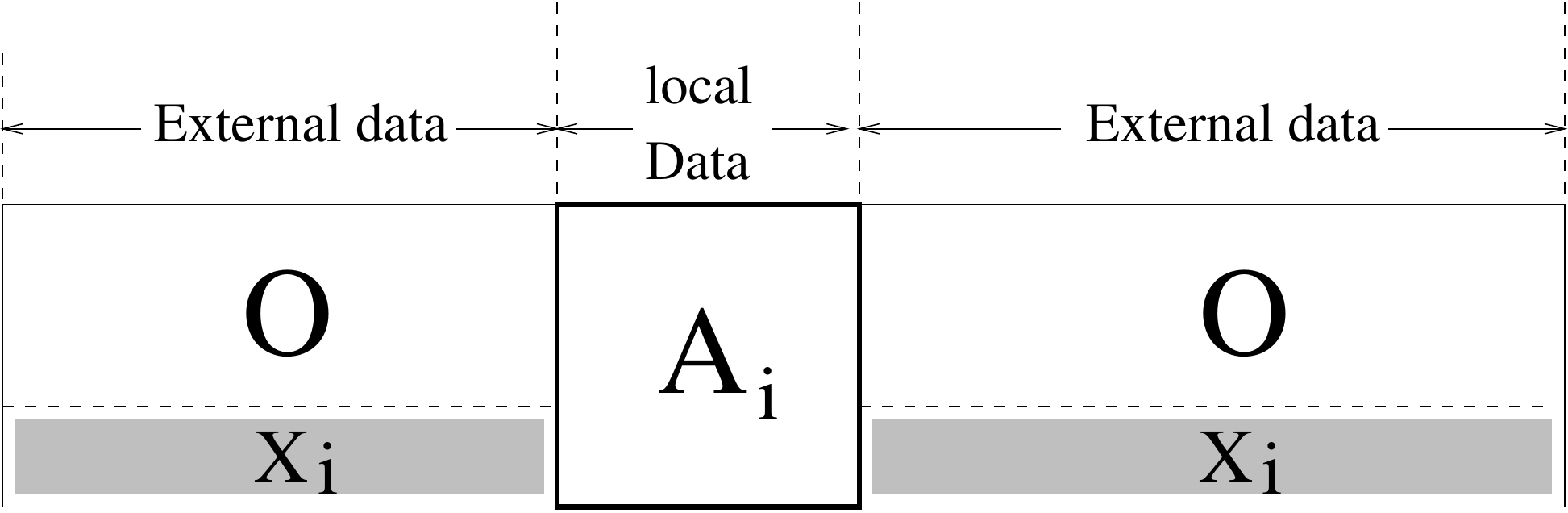}
}
\end{center}
\end{figure}

For both types of partitionings, the rows of the matrix assigned to subdomain $i$ can be split into two
parts: a local  matrix $A_i$ which acts on the  local unknowns and an
interface  matrix   $X_i$  which   acts  on  the   external 
interface
unknowns (shared unknowns for edge-based partitioning).
 Local unknowns in each subdomain are reordered so that the
interface unknowns  are listed  after the interior ones.  Thus,
each vector of local unknowns $x_i$ is split into two parts: a subvector
$u_i$ of  the internal components  followed by a subvector $y_i$  of the
  interface  components.   The  right-hand-side vector  $b_i$  is
conformingly split into  subvectors $f_i$ and $g_i$.
Partitioning the  matrix according to this splitting,  the local system
of equations can be written as
%\vspace{-10pt}
\begin{equation}\label{eq:locsys2}
%\underbrace{
\xpmatrix{
 B_i  &   E_i \cr 
 E_i^T  &   C_i }
 %}_{A_i} 
%\underbrace{ 
 \xpmatrix{ u_i \cr y_i}
 %}_{x_i} 
 + 
 \xpmatrix{ 0 \cr     \sum_{j \in N_i} E_{ij} y_j } 
 = 
 %\underbrace{
 \xpmatrix{ f_i \cr g_i}.
 %}_{b_i}.
\end{equation}
Here $N_i$  is   the set  of the indices of the
subdomains that are neighboring  to   $i$.
The term  $E_{ij} y_j$ is a part of  the product 
which reflects the contribution  to the local equation from
the neighboring  subdomain  $j$. 
The  result of  this multiplication affects
only local interface equations, which is indicated by the zero in 
the top part of the second  term  of   the left-hand side    of
\nref{eq:locsys2}.
If we denote by $\mathcal{Y}_i$ the set of the local interface unknowns of subdomain $i$, 
%(the unknowns which correspond to $y_i$)
then the global interface  $\mathcal{Y}$ is given by $\mathcal{Y}=\bigcup_{i=1}^{p} \mathcal{Y}_i$, and let $y$ and $g$ be the subvectors of $x$ and $b$ corresponding to $\mathcal{Y}$.  Note that in the case of the vertex-based partitioning, we have $\mathcal{Y}_i \cap \mathcal{Y}_j = \varnothing$, for $i \neq j$ such that $y^T=[y_1^T, y_2^T, \cdots, y_p^T]$ and $g^T=[g_1^T, g_2^T, \cdots, g_p^T]$.
If we stack all interior unknowns $u_1, u_2, \ldots, u_p$
into a vector $u$ in this order, and
we reorder the equations so that
$u$ is listed first followed by $y$, we obtain
a global system that has the following form:
\begin{equation} \label{eq:Global}
\def\arraystretch{1.1}
\left(
\begin{array}{cccc|c}
B_1      &        &   &    &  E_{1} \cr
         & B_2    &   &      & E_{2} \cr
   &        &   \ddots &  &  \vdots \cr
         &        &   &   B_p       & E_p \cr
         \hline       
 E_1^T    & E_2^T   &   \ldots &  E_p^T   &   C      \cr
\end{array}
\right)
\xpmatrix{
u_1 \cr u_2 \cr \vdots \cr u_p \cr y} =
\xpmatrix{
f_1 \cr f_2 \cr \vdots \cr f_p \cr g},
\end{equation}
or a more compact form,
\begin{equation} \label{eq:Global2}
\begin{pmatrix}B   &  E \cr  E^T & C \end{pmatrix}
\begin{pmatrix} u \cr y \end{pmatrix} 
=
\begin{pmatrix} f \cr g \end{pmatrix}.
\end{equation}

An illustration is shown in Figure~\ref{fig:Bmat} for the vertex-based
and the edge-based partitionings of $4$ subdomains for a 2-D Laplacian
matrix. Each of these two partitioning methods has its advantages and
disadvantages.  In the present work, we will focus on
the edge-based  partitioning, but  this  approach is
also applicable to the situation of a vertex-based partitioning.

\begin{figure}[th!]
\caption{An example of a 2-D Laplacian matrix which is partitioned into $4$ subdomains with  edge separators (left) and vertex separators (right), respectively. \label{fig:Bmat}}
\begin{center}
\subfigure{
\includegraphics[scale=0.5]{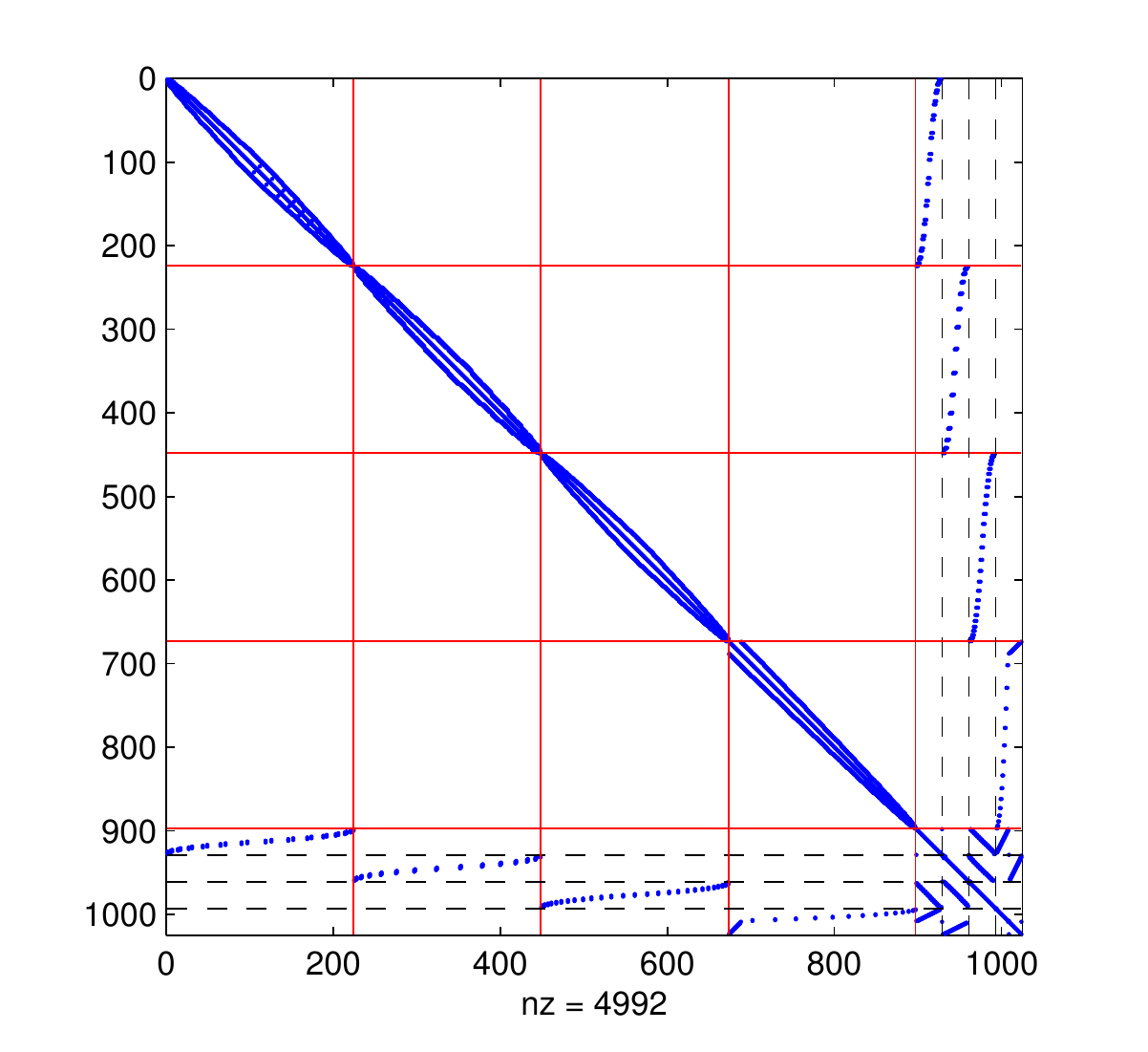}
}
\subfigure{
\includegraphics[scale=0.5]{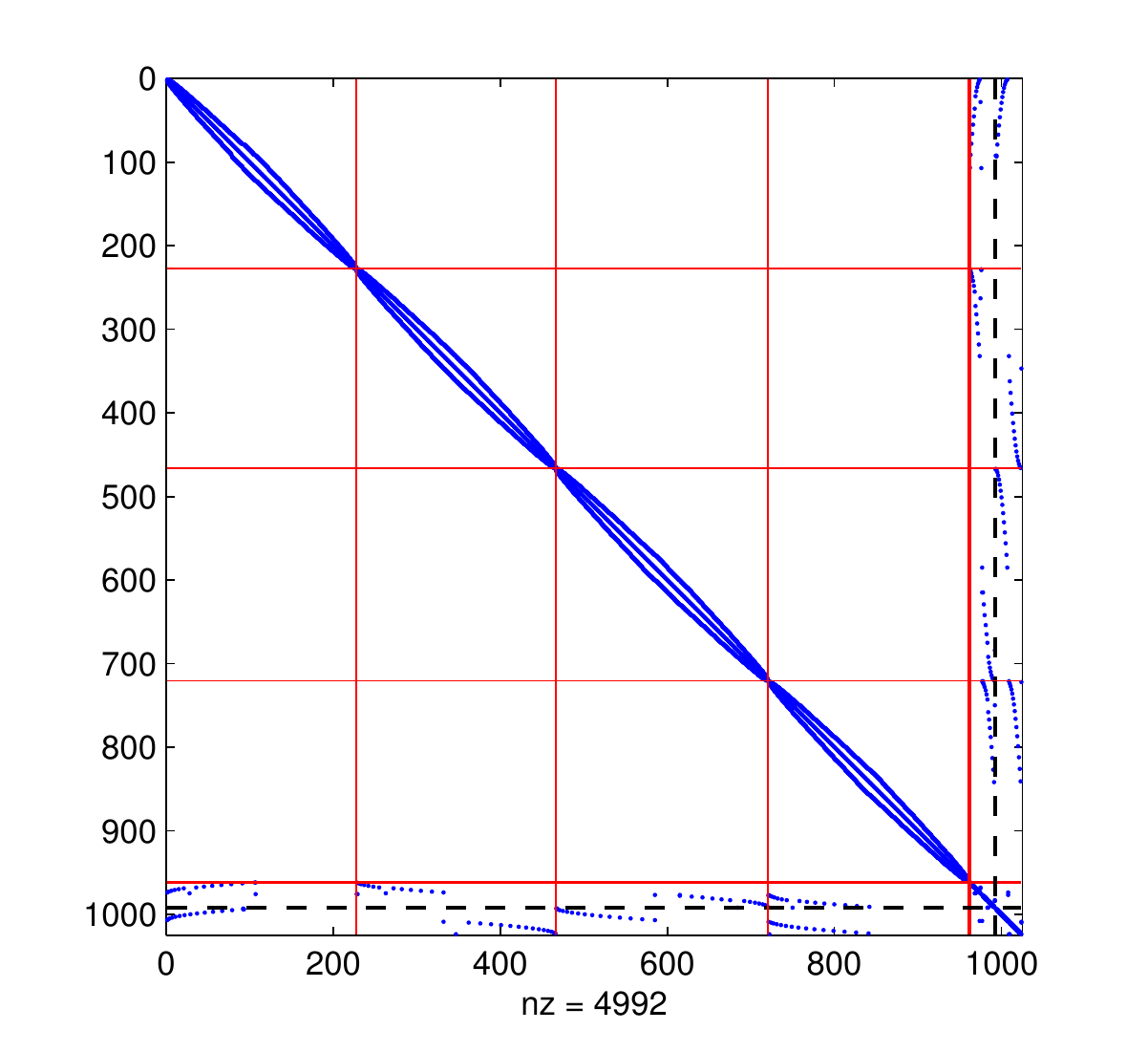}
}
\end{center} 

\end{figure}

A popular way of solving a global matrix in the form
of \eqref{eq:Global2} is to exploit the Schur complement techniques 
that eliminate the interior unknowns $u_i$ first and then focus on computing 
in some way the interface unknowns. 
A novel approach based on this principle is proposed in the next section.

\subsection{Schur complement techniques}
\label{sec:decay} 
To solve the system \eqref{eq:Global2} obtained from a DD reordering,
a number of techniques rely on 
the following basic block factorization 
\eq{eq:Schur}
\begin{pmatrix} B   &  E \cr  E^T & C \end{pmatrix} = 
\begin{pmatrix} I & \cr  E^T B\inv &  I \end{pmatrix} 
\begin{pmatrix} B  &  E   \cr  &  S \end{pmatrix} 
\quad \mbox{with} \quad  S = C - E^T B\inv E , 
\en
where $S\in \RR^{s\times s}$ is  the `Schur complement' matrix.
If an approximate solve with  the matrix $S$ is available then
one can easily solve the original system by exploiting  the above factorization. In this case note that
this will require two solves with $B$ and one solve with $S$.
In  classical ILU-type  preconditioners, e.g., in  a two-level  ARMS
method~\cite{Saad-Suchomel-ARMS}, an approximation to the   Schur complement
 $S$ is  formed by  dropping small terms and then an ILU  factorization 
of $S$ is obtained. In contrast,  the  SLR preconditioner
introduced in this  paper \emph{approximates the inverse of  $S$ directly}
by the  sum of  $C^{-1}$ and  a low-rank  correction term, resulting
in improved  robustness for  indefinite problems.   Details on
the  low-rank property for  $S^{-1}-C^{-1}$ will  be discussed  in the
next section.

\section{Spectral analysis} \label{sec:specanaly}
In this section we study the fast eigenvalue
 decay property of $S\inv - C\inv$.
In other words, our goal is to show that 
$S\inv \approx C\inv + \mathrm{LRC}$, where LRC stands for 
low-rank correction matrix.

\subsection{Decay properties of $\mathbf{S\inv -C\inv}$}
\label{sec:SinvC}
Assuming that the matrix $C$ in \eqref{eq:Global} is SPD and 
%\begin{equation}
$C=LL^T$
%\end{equation}
is its Cholesky factorization, then we can write
\begin{equation} \label{eq:S}
 S = L\left(I-L\inv E^T B\inv EL\invt\right)L^T 
\equiv L (I-H) L^T.
\end{equation}
Consider now  the spectral factorization of $H\in \RR^{s\times s}$
\begin{equation} \label{eq:H}
H = L\inv E^T B\inv E L\invt = U \Lambda U^T,
\end{equation}
where $U$ is unitary,
and $\Lambda = \mathrm{diag}\left(\lambda_1,\ldots,\lambda_s \right)$ 
is the diagonal matrix of  eigenvalues.
When $A$ is SPD, then $H$ is at least Symmetric 
Positive Semi-Definite (SPSD) and the
following lemma
shows that the eigenvalues $\lambda_i$'s are all less than one.
\begin{lemma} \label{lem:eigH}
Let
$H=L\inv E^T B\inv E L\invt$ and assume that $A$
is SPD. Then we have $0\leq\lambda_i<1$,
for each eigenvalue  $\lambda_i$ of $H$, $i=1,\ldots,s$.
\end{lemma}
\begin{proof}
If $A$ is SPD, then $B$, $C$ and $S$ are all SPD.
Since an arbitrary eigenvalue $\lambda(H)$ of $H$ satisfies
\[
\lambda(H)=\lambda(C\inv E^TB\inv E)=\lambda(C\inv(C-S))=1-\lambda(C\inv S) <
1,
\]
and $H$ is at least SPSD, we have $0\leq\lambda_i<1$.
\end{proof}

From \eqref{eq:S}, we know that the inverse of $S$ reads
\begin{equation}
S\inv = L\invt(I-H)\inv L\inv. \label{eq:Sinv1} 
\end{equation}
Thus, we wish to show that \emph{the matrix $(I-H)\inv $ can be well approximated by an identity matrix plus a low rank matrix}, from which it would
follow that  $S\inv \approx C\inv + \mathrm{LRC}$ as desired.
We have the following relations,
\eq{eq:Diff1}
(I-H)\inv - I = 
L^T S\inv L - I  = 
L^T (S\inv -C\inv ) L \equiv X,
\en
from which we obtain:
\eq{eq:SinvCinv}
S\inv = C\inv + L\invt X L\inv . 
\en
Note that the  eigenvalues of $X$ are the same as  those of the matrix
$S\inv  C -  I$.  Thus,  we will  ask the  question: Can  $X$  be well
approximated  by  a  low rank  matrix?  The  answer  can be  found  by
examining the  decay properties  of the eigenvalues  of $X$,  which in
turn  can be  assessed by  checking the  rate of  change of  the large
eigenvalues of $X$.  We can state the following result.
\begin{lemma}
The matrix $X $ in \nref{eq:Diff1} has the nonnegative eigenvalues
$\theta_k = \lambda_k /(1-\lambda_k)$
for $k=1,\cdots, s$, where $\lambda_k$ is the eigenvalue
of the matrix $H$ in \eqref{eq:H}.
\end{lemma} 
\begin{proof}
From \nref{eq:Diff1} the eigenvalues of the matrix $X$ are 
$(1-\lambda_k)\inv -1  = \lambda_k/(1-\lambda_k)$. These are nonnegative
because from Lemma~\ref{lem:eigH} the $\lambda_k$'s are between  0 and 1.
\end{proof}
%%$ \theta_k = (1-\lambda_k)\inv -1 = \lambda_k /(1-\lambda_k) $,

Now we consider the derivative of $\theta_k$ with respect to $\lambda_k$:
\[
\frac{d {\theta_k}}{d {\lambda_k}}=\frac{1}{(1- \lambda_k)^2} \ .
\]
This indicates a rapid increase when the $\lambda_k$ increases toward one. In other words, this means that the largest eigenvalues of $X$
tend to be well separated and 
$X$ can be approximated accurately by a low-rank matrix in general. 
Figure  \ref{fig:decay} illustrates  the decay of the  eigenvalues of the matrix  $L\invt X L\inv$  and the matrix $X$ for a 2-D Laplacian matrix, which is precisely  the matrix shown in Figure~\ref{fig:Bmat}.
As can be seen, using just a few eigenvalues and eigenvectors will represent 
the matrix $X$ (or $L\invt X L\inv)$ quite well.
In  this particular situation, $5$ eigenvectors (out of the total of $127$) will capture 
$82.5\% $ of $X$ and $85.1\% $ of $L\invt X L\inv$, 
whereas 10 eigenvectors will capture $89.7\%$ of $X$ and 
$91.4\% $ of $L\invt X L\inv$.

\begin{figure}[h!b]
\centering
%\framebox{
\begin{minipage}[b][][c]{0.5\columnwidth}
\centering
\includegraphics[scale=0.45]{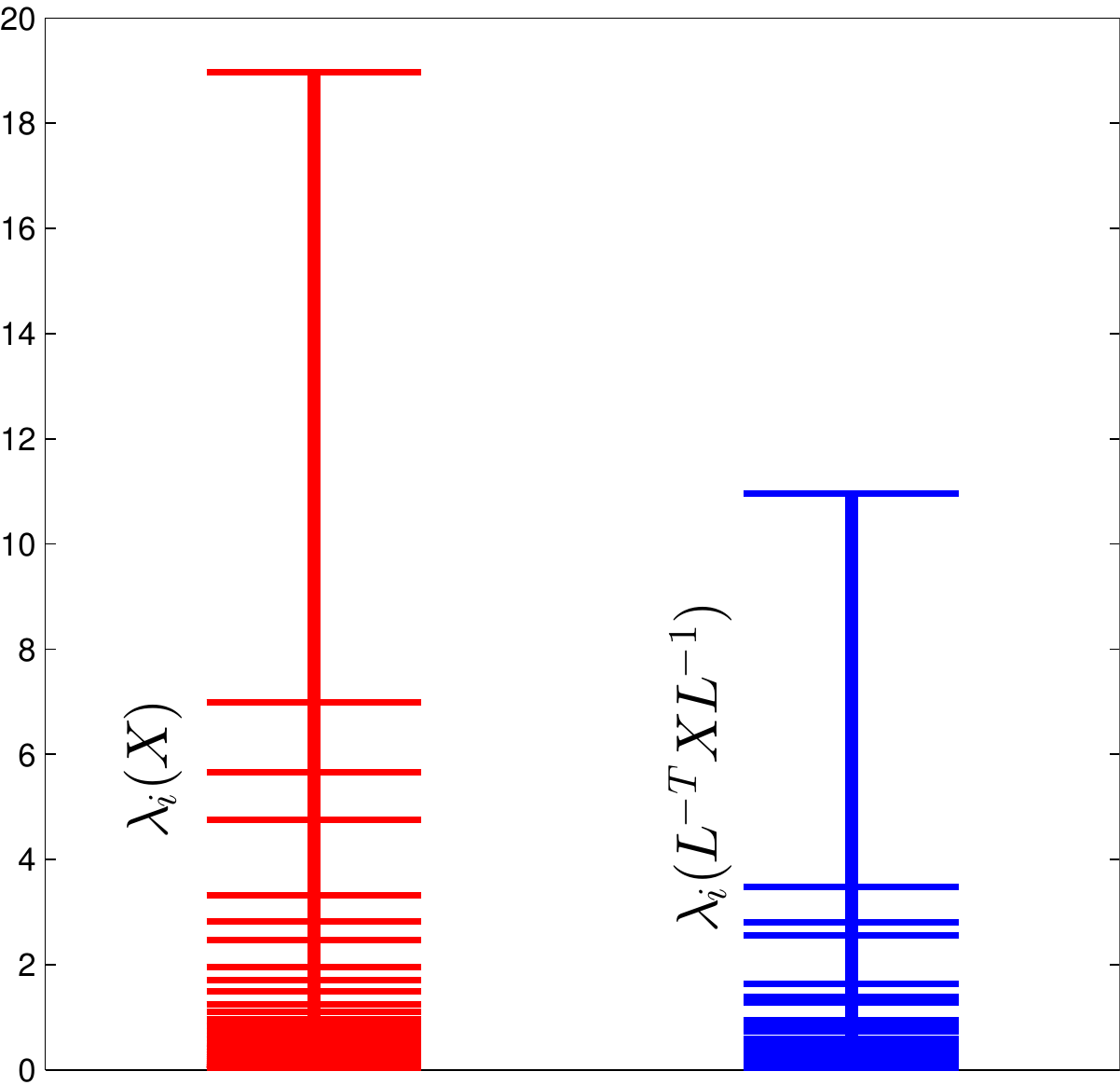}  
\end{minipage}%
\hspace{1em}
%\framebox{
\begin{minipage}[b][][c]{0.4\columnwidth}
\centering
\caption{Illustration of 
the decay of eigenvalues of $X$ (left) and $S\inv - C\inv = L\invt X L\inv$  (right) for a 2-D Laplacian matrix with $n_x=n_y=32$, where the  domain is  decomposed into  $4$  subdomains (i.e., $p=4$), and the size of $S$ is $127$. $5$ eigenvectors  will capture 
$82.5\% $ of the spectrum of $X$ and $85.1\% $ of the spectrum of $L\invt X L\inv$, 
whereas $10$ eigenvectors will capture $89.7\%$ of the spectrum of $X$ and $91.4\% $ of the spectrum of $L\invt X L\inv$.}
\end{minipage}%
\label{fig:decay}
\end{figure} 

\begin{comment}
In Figure \ref{fig:decay}, we show the eigenvalues $\lambda_k$'s of $H$ and the fast decay of the eigenvalues $\theta_k$'s of $X$ for the case of a 2-D Laplacian matrix.
%the same case in Figure~\ref{fig:eiggamma}.
\begin{figure}[ht!]
\center
\caption{Illustration of the decay of the eigenvalues $\theta_k$ of  $X$, and the eigenvalues $\lambda_k$ of $H$, for $-\Delta$ on a $65 \times 65$ grid, and $p=2$. \label{fig:decay}}
\vspace{.5em}
\includegraphics[scale=0.4]{eig652DX}
\end{figure}
\end{comment}

\subsection{Two-domain analysis in a 2-D model problem} \label{sec:twodomanaly}
The spectral analysis  of the matrix $S\inv -  C\inv$ is difficult for
general problems and general  partitionings.  In the simplest case when
the  matrix $A$ originates  from a  2-D Laplacian  on a  regular grid,
discretized by  centered differences, and it  is partitioned  into $2$
subdomains, the  analysis becomes feasible.  The goal  of this section
is  to show  that the  eigenvalues  of $X$  and $L^{-T}XL^{-1}$  decay
rapidly.

Assume that $-\Delta$  is discretized on a grid  $\Omega$ of size $n_x
\times  (2n_y+1)$  with Dirichlet  boundary  conditions  and that  the
ordering is  major along the  $x$ direction.  The grid  is partitioned
horizontally into  three parts: the two disconnected  $n_x \times n_y$
grids, namely $\Omega_1$  and $\Omega_2$, which are the  same, and the
$n_x   \times  1$   separator   denoted  by   $\Gamma$.   See   Figure
\ref{fig:2danaly}(a) for  an  illustration.   Let   $T_x$  be  the
tridiagonal matrix corresponding to  $\Gamma$ of dimension $n_x \times
n_x$  which discretizes  $ - \partial^2/\partial x^2$.  The  scaling term
$1/h^2$ is  omitted so  that $T_x$  has the constant  $2$ on  its main
diagonal and $-1$ on the co-diagonals. 
%%We will call $I_x$ the identity matrices of size  $n_x$. 
Finally, we denote by  ${A}$ the matrix which
results  from discretizing  $-\Delta$  on $\Omega$  and 
reordered  according      to     the partitioning 
$\Omega=\left\{\Omega_1,\Omega_2,\Gamma\right\}$. In $\Omega_1$ and $\Omega_2$,  the interface nodes are ordered at the end. Hence, ${A}$ has the
form: 
\begin{equation} \label{eq:matA}
{A} =\begin{pmatrix}
A_{y}   &      & E_{y} \\
     & A_{y}   & E_{y} \\
E_{y}^T & E_{y}^T & \hat{T}_x
\end{pmatrix},
\end{equation}
where $A_{y}$ corresponds to the $n_x \times n_y$ grid (i.e., $\Omega_1$ or $\Omega_2$), $E_{y}$ defines the couplings between $\Omega_1$ (or $\Omega_2$) and $\Gamma$, and the matrix $\hat{T}_x$ is associated with $\Gamma$, for which we have
\begin{equation}
\hat{T}_x = T_x + 2 I.
\end{equation}
%Figure \ref{fig:2danalya} shows the partitioned mesh and
Figure \ref{fig:2danaly}(b) is an illustration of the nonzero pattern of $A$.

%An illustration of  $A_{y-k}$, $E_{y-k}$ and 
%% the corresponding mesh is shown in Figure \ref{fig:2danaly}
\begin{figure}[h!t]
\caption{ Illustration of the matrix $A$ and the corresponding partitioning of the 2-D mesh. \label{fig:2danaly}
}
%\vspace{.5em}
\begin{center}
\subfigure[Partition of a regular mesh into $3$ parts.]{
\includegraphics[scale=0.33]{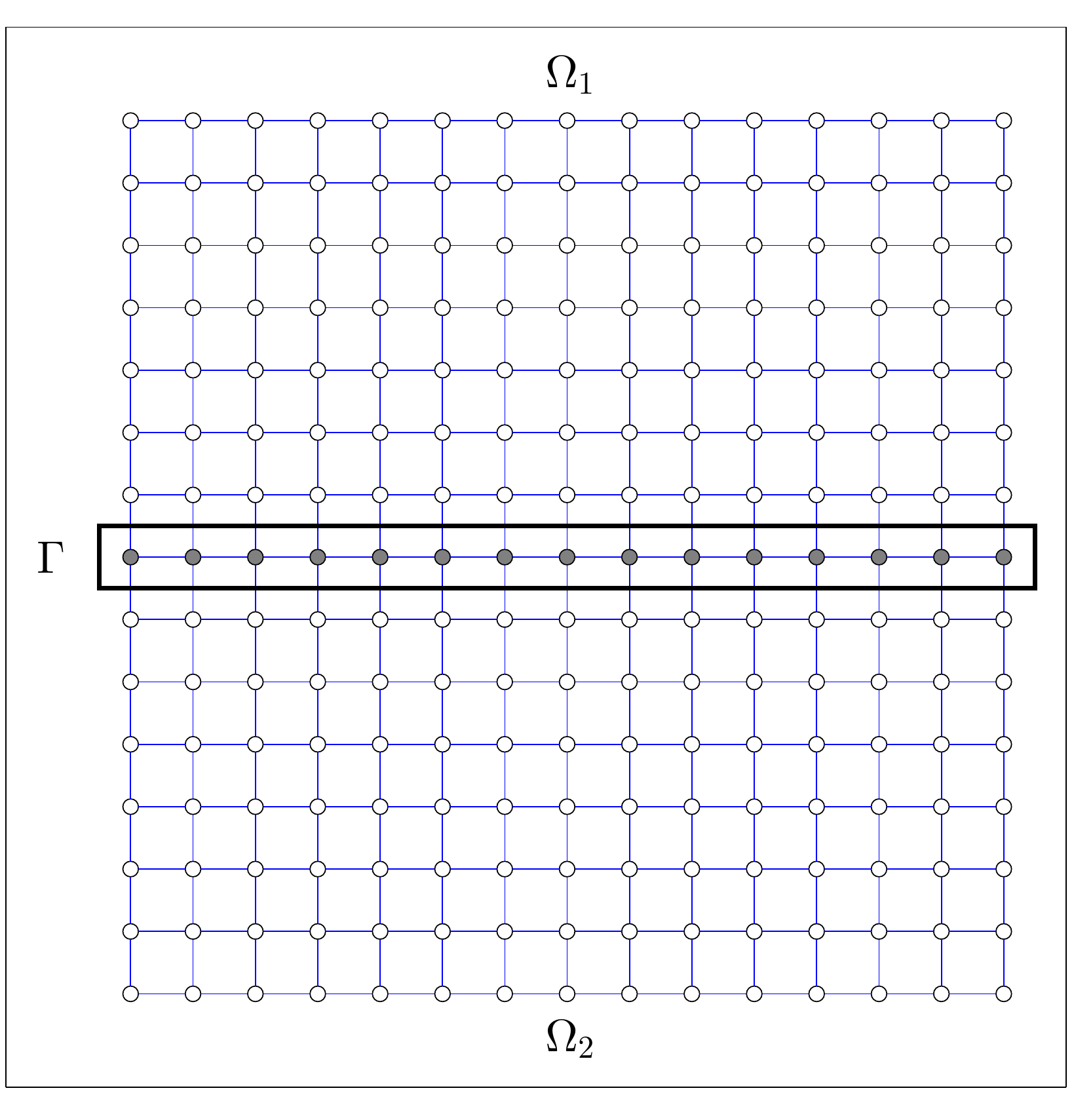}
}
\hspace{2em}
\subfigure[Nonzero pattern of the reordered matrix.]{
\includegraphics[scale=0.34]{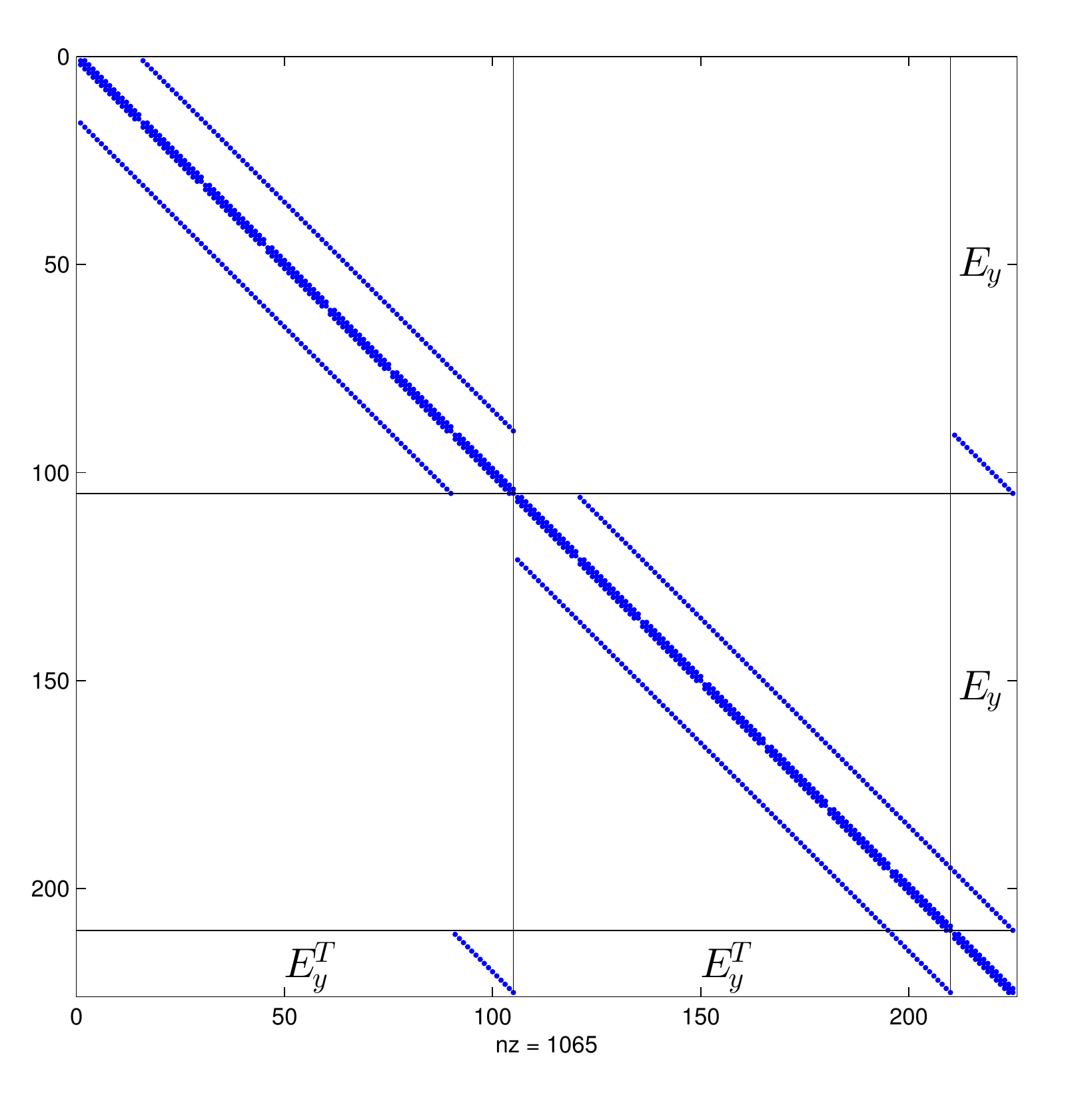}
} 
\end{center}
\end{figure}

Therefore, the Schur complement associated with $\Gamma$ in \eqref{eq:matA} reads
\begin{equation} \label{eq:SchurGamma}
S_{\Gamma} = \hat{T}_x - 2E_{y}^T A_{y}\inv E_{y} \ ,
\end{equation}
and the eigenvalues of $X$ and $L^{-T}XL^{-1}$ 
%in (\ref{eq:SinvCinv}) 
correspond to those of $S_{\Gamma}^{-1}\hat{T}_x - I$ and $S_{\Gamma}^{-1} - \hat{T}_x^{-1}$, respectively, in this case.
The coupling matrix $E_y$ has the form $E_y^T=(0, I_x)$, where $I_x$ denotes the identity matrix of size $n_x$. Clearly, the matrix $R_y =E_y^T A_y \inv E_y$ is simply the bottom right (corner) block of the 
inverse of $A_y$, which can be readily obtained from a standard block factorization.  
%%The matrix $R_y$ can be seen to be
%% a rational function of $\hat T_x$.
Noting that $A_y$ is of the form
\[
A_y = 
\begin{pmatrix} 
\hat T_x &   -I    &           &                   \cr
-I  &   \hat T_x   &  -I       &                   \cr
    & \ddots  &  \ddots   & \ddots            \cr
    &         &  \ddots   & \ddots  & -I      \cr
    &         &           &  -I     &  \hat T_x         
\end{pmatrix},\]
we write its  block LU factorization  as:
\[
A_y =
\begin{pmatrix} 
  \hspace{10pt}  I    &         &     &                   \cr
-D_1 \inv     &   I     &           &                   \cr
              & -D_2\inv   &  \ddots   &                   \cr
    &         &  \ddots   & \ddots  &         \cr
    &         &           &  - D_{n_y}\inv     &  I           
\end{pmatrix} 
\begin{pmatrix} 
 D_1  &  - I       &           &                   \cr
       &   D_2     &  -I       &                   \cr
    &         &  \ddots   & \ddots            \cr
    &         &           & \ddots  &  -I     \cr
    &         &           &     &   D_{n_y}     
\end{pmatrix}  . 
\]
The $D_i$'s satisfy the recurrence: 
$D_{k} = \hat T_x - D_{k-1}\inv $, for $k=2, \cdots, n_y$ starting
with $D_{1} = \hat T_x$. 
The result is that each $D_k$ is a continued fraction in $\hat T_x$. 
As can be easily verified $R_y $ is equal to $D_{n_y}\inv$.
The scalar version of the above recurrence is of the form:
\[
d_k = 2 a - \frac{1}{d_{k-1}} \ , \quad k=2,\cdots, n_y \ , \quad \mbox{with} \quad 
d_1 \equiv 2a \ . 
\]
The $d_i$'s are the diagonal entries of the U-matrix of an LU factorization similar to the one
above but applied to the $n_y \times n_y $
tridiagonal matrix $T$ that has
$2 a$ on the diagonal and $-1$ on the co-diagonals. 
For reasons that will become clear we replaced the matrix $\hat T_x $ by the scalar $2 a$.
We are interested in the inverse of the last entry, i.e.,   $d_{n_y} \inv$.
Using Chebyshev polynomials we can easily see that
$d_{n_y}\inv = U_{n_y-1} ( a  ) / U_{n_y} ( a  ) $ where 
$U_k(t)$ is the Chebyshev polynomial of the second kind (for details, see  Appendix):
\[
U_k(t) =
\frac{
\sinh ( (k+1) \cosh \inv(t))} 
{\sinh(\cosh \inv(t))} \ .
\]
In terms of the original matrix $A_y$, the scalar $a$ needs to be substituted by
$ \hat T_x/2 = I + T_x/2$.
In the end,
the matrix $S\inv -C\inv = S_\Gamma\inv - \hat T_x\inv$ is a rational 
function of $ \hat T_x/2 $. We denote this rational function  by $s(t)$, i.e.,
$S_\Gamma\inv -\hat T_x\inv = s(\hat T_x/2)$ and note that $s$ is well-defined in terms
of the scalar $a$. Indeed, from the above: 
\[
s(a) = 
\frac{1}{2a - 2 \frac{U_{n_y-1} (a)}{U_{n_y} (a)}} 
- \frac{1}{2a}  
= \frac{U_{n_y-1} (a)}{ a (2 a U_{n_y} (a) - 2 U_{n_y-1} (a)) }
= \frac{U_{n_y-1} (a)}{ a \left[ U_{n_y+1} (a) - U_{n_y-1} (a) \right] }  .
\] 
Everything can now be expressed in terms of the eigenvalues
of $\hat T_x/2$ which are %given by
\eq{eq:etak0}
\eta_k = 1 + 2 \sin^2 \frac{k \pi }{2 (n_x + 1)} \ ,
\quad k=1,\cdots, n_x \ .
\en 
We can then state the following.

\begin{proposition} 
Let $\eta_k$ be defined in \nref{eq:etak0} and
$\theta_k  = \cosh\inv( \eta_k) $, $k=1,\cdots, n_x$.  Then,
the eigenvalues $\gamma_k$ of $ S_\Gamma\inv - \hat T_x\inv$ 
are given by 
\eq{eq:gammak0}
\gamma_k  = 
\frac{\sinh (  n_y  \theta_k ) }
{ \eta_k \left[ \sinh ((n_y+2) \theta_k ) 
- \sinh (n_y \theta_k) \right] } \ ,  \quad  k=1,\cdots, n_x \ .
\en 
\end{proposition}

Note that we have $e^{\theta_k} = \eta_k + \sqrt{\eta_k^2 -1} $ and
$\sinh (n \theta_k) = [(\eta_k + \sqrt{\eta_k^2 -1} )^{n} -  
(\eta_k + \sqrt{\eta_k^2 -1} )^{-{n}}]/2$, which is well
approximated by 
$(\eta_k + \sqrt{\eta_k^2 -1})^{n}/2$ for a large $n$.
In the end, assuming $n_y$ is large enough, we have 
\begin{equation} \label{eq:gamma}
\gamma_k \approx 
\displaystyle\frac{1}
{ \eta_k \left[(\eta_k + \sqrt{\eta_k^2 -1})^2 -1\right] } 
= 
\frac{1}
{ 2 \eta_k \left[(\eta_k^2-1) + \eta_k \sqrt{\eta_k^2 -1} \right] } \ .
\end{equation}
This shows that for those eigenvalues of $\hat T_x$ that are close to
one, we would have a big amplification to the value $1/\eta_k$.
These eigenvalues correspond to the smallest eigenvalues of $T_x$.
We can also show that
\[
\gamma_k \approx 
\frac{1}{2} \left[
\frac{1} {\sqrt{\eta_k^2 -1} } - 
\frac{1} {\eta_k} \right],
\] 
and for the eigenvalues $\zeta_k$ of $S_\Gamma\inv \hat{T}_x -I$, we have
\[
\zeta_k = 2\eta_k\gamma_k \approx \frac{\eta_k}{\sqrt{\eta_k^2-1}}-1 \ .
\]
An illustration of $\gamma_k$, $\zeta_k$ and $1/\eta_k$ is shown in Figure \ref{fig:eiggamma}.

\begin{figure}[h!t]
\caption{Illustration of the decay of the eigenvalues $\gamma_k$ of the matrix $S\inv-C\inv$ and the eigenvalues $\zeta_k$ of the matrix $S\inv C-I$, and $1/\eta_k$ for $-\Delta$ on a 2-D grid of size $n_x \times (2n_y+1)$ with $n_x=65,n_y=32$, which is partitioned into $2$ subdomains. \label{fig:eiggamma}}
\begin{center}
\subfigure{
\includegraphics[scale=0.4]{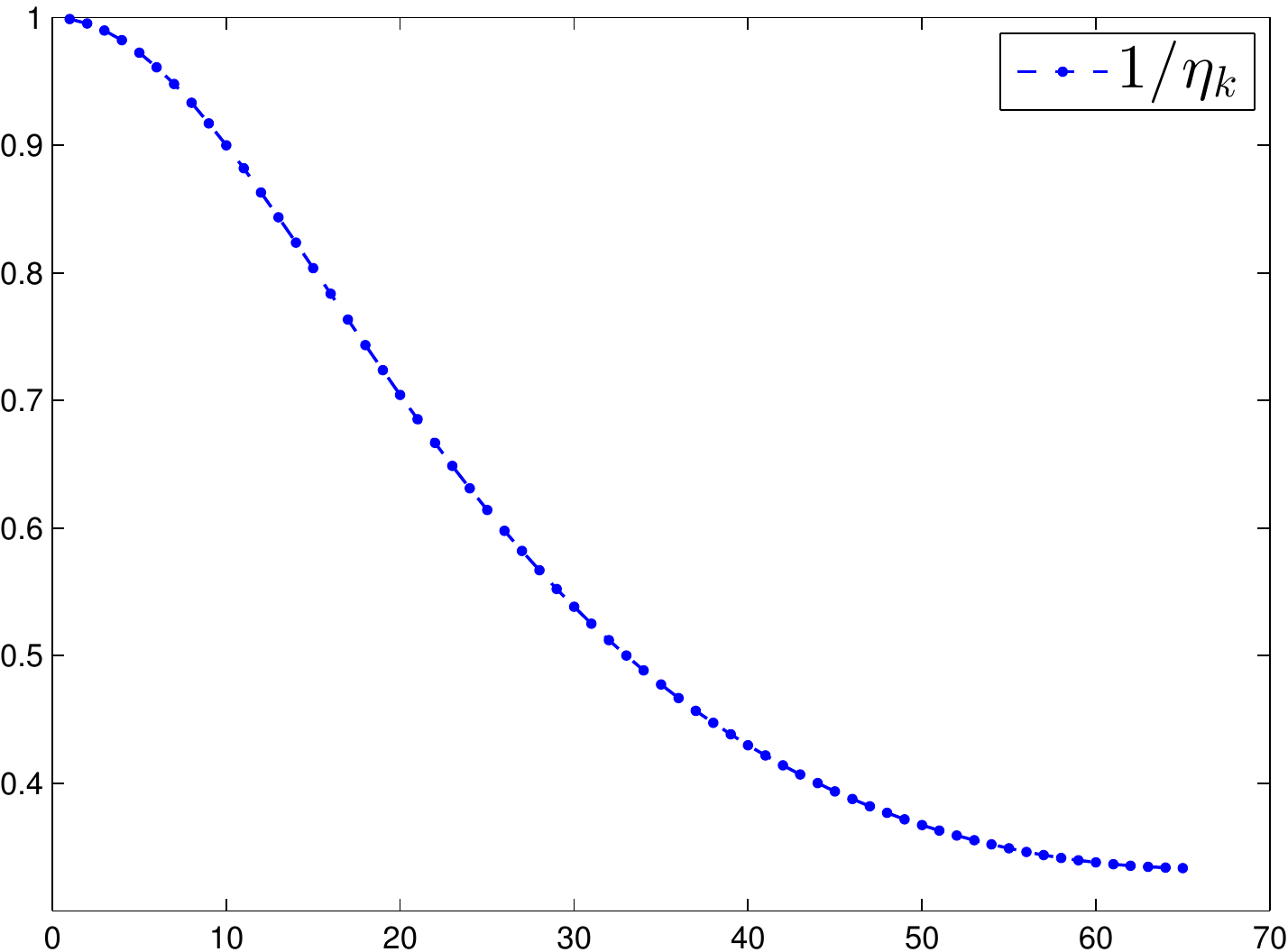}
}
\hspace{1em}
\subfigure{
\includegraphics[scale=0.4]{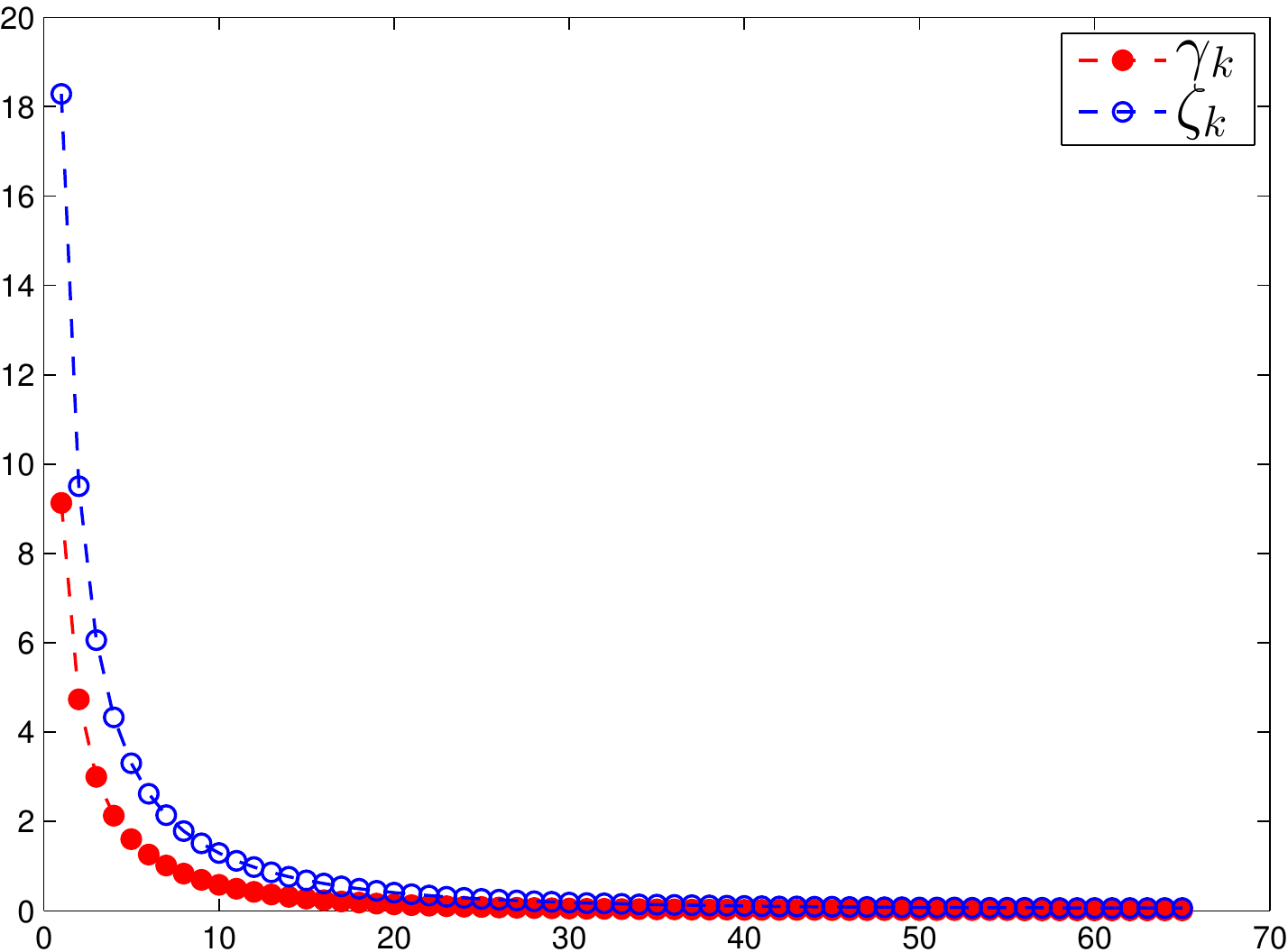}
}
\end{center}
\end{figure}

\section{Schur complement based preconditioning with low rank corrections}\label{sec:LowRank}

The goal of this section is  to build a preconditioner for a matrix of
the  form  \nref{eq:Global2} as  obtained  from  the  DD method.   The
preconditioning matrix $M$ is of the form
\begin{equation} \label{eq:M}
M = 
\begin{pmatrix} I & \cr  E^T B\inv &  I \end{pmatrix} 
\begin{pmatrix} B  &  E   \cr  &  \tilde{S} \end{pmatrix},
\end{equation}
where $\tilde{S}$ is an approximation  to $S$.  
The above is approximate factorization of \nref{eq:Schur} whereby 
(only) $S$ is approximated. 
%%A first consequence of
%%\eqref{eq:M} is that  $m$ eigenvalues of $AM\inv$ will  take the value
%%one  independently of  the choice  of  $\tilde{S}$, where  $m$ is  the
%%!dimension of  the matrix $B$, and  the remaining $s$ ones  will be the
%%e!!eigenvalues of the matrix $S\tilde{S}\inv$. 
In fact we will approximate directly the inverse of $S$ instead of $S$
by exploiting low-rank properties.
Specifically, we seek an approximation of the form $\tilde{S}\inv=C\inv +
\mathrm{LRC}$.  From a  practical point of view, it  will be difficult
to compute directly an approximation   to the matrix  $S\inv -C\inv$,
since we do not  (yet) have an efficient means for solving
linear systems with the matrix $S$.   
Instead we will extract this approximation from that of the matrix 
$X$ defined  
%%Therefore,  we will  adopt  the  low-rank approximation  scheme
%%discussed   
in  Section   \ref{sec:SinvC}, see 
\eqref{eq:Diff1}.    Recall  the   expression  \eqref{eq:S}   and  the
eigen-decomposition of $H$ in \eqref{eq:H}, which yield,
\begin{equation} \label{eq:slr_S}
 S = L(I-U \Lambda U^T)L^T = LU(I- \Lambda )U^TL^T.
\end{equation}
% \[ 
%  S = L(I-U \Lambda U^T)L^T,
% \]
% or 
% \eq{eq:LinvSLinvt} 
% L\inv  S  L\invt =U (I-\Lambda)  U^T. 
% \en 
The inverse of $S$ is then
\begin{equation}
S\inv = L\invt  U (I-\Lambda)\inv U^T  L\inv,
\end{equation}
which we write in the form,
\begin{equation} \label{eq:Sinv}
 S\inv = L\invt
\left(I+U [\left(I-\Lambda\right)\inv-I ] U^T\right)L\inv.
\end{equation}
Now, assuming  that $H$ has an approximation of the following form,
\begin{equation} \label{eq:Happrox}
\tilde{H} \approx U \tilde \Lambda U^T,
\end{equation}
we will obtain the following approximation to $S\inv$:
% \begin{align} 
%  \tilde{S}\inv &= L\invt
% \left(I+U\left(\left(I-\tilde{\Lambda}\right)\inv-I\right)U^T\right)L\inv \label{eq:tildeSinv} \\
% &= C\inv + L\invt U \left( (I-\tilde{\Lambda})\inv - I \right) U^T L\inv. \label{eq:tildeSinv2}
% \end{align}
\begin{align}
\tilde S\inv &= L\invt  U (I-\tilde \Lambda)\inv U^T  L\inv,
\label{eq:tildeSinv} \\
%&= L\invt [I+U[(I-\tilde{\Lambda})\inv-I]U^T]L\inv, \notag \\
&= C\inv + L\invt U [ (I-\tilde{\Lambda})\inv - I ] U^T L\inv.
\label{eq:tildeSinv2}
\end{align}

\begin{proposition} \label{prop:eigSS}
Let $S$ and $H$ be defined by \nref{eq:S} and \nref{eq:H} respectively and 
let $\Sigma=\mathrm{diag}(\sigma_1,\ldots,\sigma_s)$ with the
$\sigma_i$'s defined by
\begin{equation} \label{eq:eigSSinv}
\sigma_i = \frac{1-\lambda_i}{1-\tilde{\lambda}_i} \ , \quad i=1,\ldots,s.
\end{equation}
Then, the eigendecomposition of $S \tilde S\inv$ is given by:
\eq{eq:DiffS}
S \tilde S\inv 
= (L  U) \Sigma (L U) \inv.  
\en
\end{proposition} 
\begin{proof}
From \eqref{eq:slr_S} and \eqref{eq:tildeSinv}, we have
\begin{align}
 S \tilde S\inv  &= LU(I-\Lambda)U^TL^T L\invt ( U (I-\tilde \Lambda)\inv
U^T ) L\inv \notag \\
&= (LU)(I-\Lambda)(I-\tilde \Lambda)\inv \notag
(U^T L\inv) = (L  U) \Sigma (L U) \inv.
\end{align}
% Subtracting the  expression \nref{eq:tildeSinv} from that of \nref{eq:Sinv} gives:
% \eq{eq:SinvDiff} 
% S\inv - \tilde{S}\inv =
% L\invt U\left(\left(I-\Lambda\right)\inv -
% \left(I-\tilde{\Lambda}\right)\inv\right)U^TL\inv.
% \en
% It is convenient to rewrite this by using a congruence transformation:
% \eq{eq:SinvDiff1} 
% L^T S\inv L - L^T  \tilde{S}\inv  L =
% U\left(\left(I-\Lambda\right)\inv -
% \left(I-\tilde{\Lambda}\right)\inv\right)U^T. 
% \en
% Multiply through on the left by $ (L^T S\inv L)\inv = L\inv S L\invt$:
% \begin{equation} \label{eq:SSinv}
% I - L\inv S \tilde S \inv L 
% = L\inv S L\invt U\left(\left(I-\Lambda\right)\inv -
% \left(I-\tilde{\Lambda}\right)\inv\right)U^T ,
% \end{equation}
% and now substitute \nref{eq:LinvSLinvt} for the term $L\inv S L\invt $ 
% in the right-hand side to get:
% \begin{equation} \label{eq:SSinv2}
% I - L\inv S \tilde S \inv L 
% = U \left[ I - (I-\Lambda)(I-\tilde \Lambda) \inv \right] U^T . 
% %% = U (I-\Lambda)(I-\tilde \Lambda) \inv  U^T 
% \end{equation}
%  We end up with
% \begin{equation} \label{eq:SSinv3}
% L\inv S \tilde S \inv L 
% = U \Sigma  U^T \quad \mbox{or} \quad
% S \tilde S\inv 
% = (L  U) \Sigma (L U) \inv.  
% \end{equation}
% This completes the proof.
\end{proof}

The simplest selection of $\tilde \Lambda$ is the one that ensures that the $k$
largest eigenvalues
of $(I-\tilde \Lambda)\inv$ match the largest eigenvalues 
of $(I- \Lambda)\inv$.
%This simply minimizes the 2-norm  of \nref{eq:diffHinv} under the assumption
%that the approximation in \eqref{eq:Happrox} is of rank $k$.
Assume that the eigenvalues of $H$ are $\lambda_1 \ge \lambda_2 \ge \cdots \ge
\lambda_s$, which means that the  diagonal entries $\tilde \lambda_i $  of
$\tilde \Lambda$ are selected such that
\begin{equation} \label{eq:Lambda-1}
\tilde \lambda_i = \left\{ \begin{array}{cl}
\lambda_i & \mbox{if} \quad  i\le k \\
 0  & \mbox{otherwise}
\end{array}.
\right.
\end{equation}
Proposition~\ref{prop:eigSS} indicates that in this case the eigenvalues of
$S\tilde{S}\inv$ are
\[
\left\{ \begin{array}{cl}
1  & \mbox{if} \quad  i\le k \\
1-\lambda_i  & \mbox{otherwise}
\end{array}.
\right.
\]
Thus,  we can infer that in this situation
$k$ eigenvalues of $S\tilde{S}\inv $ will take the value one and the other
$s-k$ eigenvalues $\sigma_i$ 
satisfy $0<1-\lambda_{k+1} \le \sigma_i < 1-\lambda_s < 1$.

Another choice for $\tilde{\Lambda}$, inspired by \cite{RliSaadMLR2},
will make the eigenvalues of $S\tilde{S}\inv $  larger than or equal to  one.
Consider defining $\tilde \Lambda$ such that
\eq{eq:main}
\tilde \lambda_i = \left\{ \begin{array}{cl}
\lambda_i & \mbox{if} \quad  i\le k \\
 \theta  & \mbox{if} \quad  i > k \\
\end{array} . 
\right.
\en 
Then, from \nref{eq:eigSSinv} the eigenvalues of
$S\tilde{S}\inv$ are
\begin{equation} \label{eq:eigDiffG}
\left\{ \begin{array}{cl}
1  & \mbox{if} \quad  i\le k \\
(1-\lambda_i) /  (1-\theta)
 & \mbox{if} \quad  i > k \\
\end{array}.
\right.
\end{equation}
%Note that for $i>k$, the eigenvalues can be made greater than or equal to one by selecting
%$\theta = \lambda_{k+1} < 1$. 
%In this case, the eigenvalue $\sigma_i$ is bounded
%by $\left[1,1/(1-\theta)\right)$.
The earlier definition of $\Lambda_k$ in \eqref{eq:Lambda-1} which
truncates the lowest eigenvalues of
$H$ to zero corresponds to selecting $\theta = 0$.
Note that for $i>k$, the eigenvalues can be made greater than  or equal to one by selecting
$ \lambda_{k+1} \le \theta < 1 $.  
In this case, the eigenvalues $\sigma_i $ for $i>k$ which are equal to
 $\sigma_i =  (1 - \lambda_i)/(1-\theta) $ belong to the 
interval 
\eq{eq:intv}
\left[1 , \quad \frac{1 - \lambda_s}{1-\theta}  \right] 
\subseteq
\left[1 , \quad \frac{ 1 }{1-\theta}  \right].  
\en 
Thus, the spectral condition number of the preconditioned matrix is $(1-\lambda_s)/(1-\theta)$.
The choice leading to the smallest 2-norm deviation is letting
$\theta = \lambda_{k+1} $. 
One question that may be asked is how does the  condition number
$ \kappa=\max \sigma_i / \min \sigma_i$ vary when $\theta $ varies between 0 and 1? 

First observe that  a general expression for the eigenvalues of  $S \tilde S\inv $ is given by \nref{eq:eigDiffG} regardless of the value of $\theta$. When $\lambda_{k+1} \le \theta <1$, we just saw that the spectral condition
number is equal to $(1-\lambda_s)/(1-\theta)$. The smallest value of this condition
number is reached when $\theta $ takes the smallest  value which,
recalling our restriction 
$\lambda_{k+1} \le \theta <1$,  is $\theta = \lambda_{k+1}$. 
There is a second  situation, which corresponds to 
when $\lambda_s \le  \theta \le \lambda_{k+1}$. Here the largest
 eigenvalue is still  $(1-\lambda_s)/(1-\theta)$ which is larger than one.
The smallest one is now smaller than one, which is 
$(1-\lambda_{k+1})/(1-\theta)$. So the condition number now is again
$(1-\lambda_s)/(1-\lambda_{k+1})$, which is independent of $\theta$ in the interval $[\lambda_s, \ \lambda_{k+1}]$.
The third and final situation corresponds to  the case
when $0 \le \theta \le \lambda_{s}$. The largest
 eigenvalue is now one, because $(1-\lambda_s)/(1-\theta) < 1$, while the smallest one is still
$(1-\lambda_{k+1})/(1-\theta)$. This leads to the condition number 
$(1-\theta)/(1-\lambda_{k+1})$ and the smallest spectral condition number for 
$\theta$ in this interval is reached when $\theta = \lambda_s$ leading to the 
same optimal condition number $(1-\lambda_s)/(1-\lambda_{k+1}) $.
This result is summarized in the following proposition.
\begin{proposition} 
The spectral condition number $\kappa(\theta) $ of $S \tilde S\inv$ is equal to
\eq{eq:kappa}
\kappa(\theta) = 
\begin{cases}
\displaystyle\frac{1-\theta}{1-\lambda_{k+1}}    & \mathrm{if} \; \theta \in [0, \lambda_s) \\[1em]
\displaystyle\frac{1-\lambda_s}{1-\lambda_{k+1}} & \mathrm{if} \; \theta \in [\lambda_s, \lambda_{k+1}] \\[1em]
\displaystyle\frac{1-\lambda_s}{1-\theta}        & \mathrm{if} \; \theta \in (\lambda_{k+1}, 1) 
\end{cases} 
\en 
It has a  minimum value of 
$(1-\lambda_s)/(1-\lambda_{k+1})$, which is reached for any $\theta $ in the second interval.
\end{proposition}

\begin{figure}[h!t]
\centering
\begin{minipage}[b][][c]{0.45\columnwidth}
\centering
\includegraphics[width=2.3in]{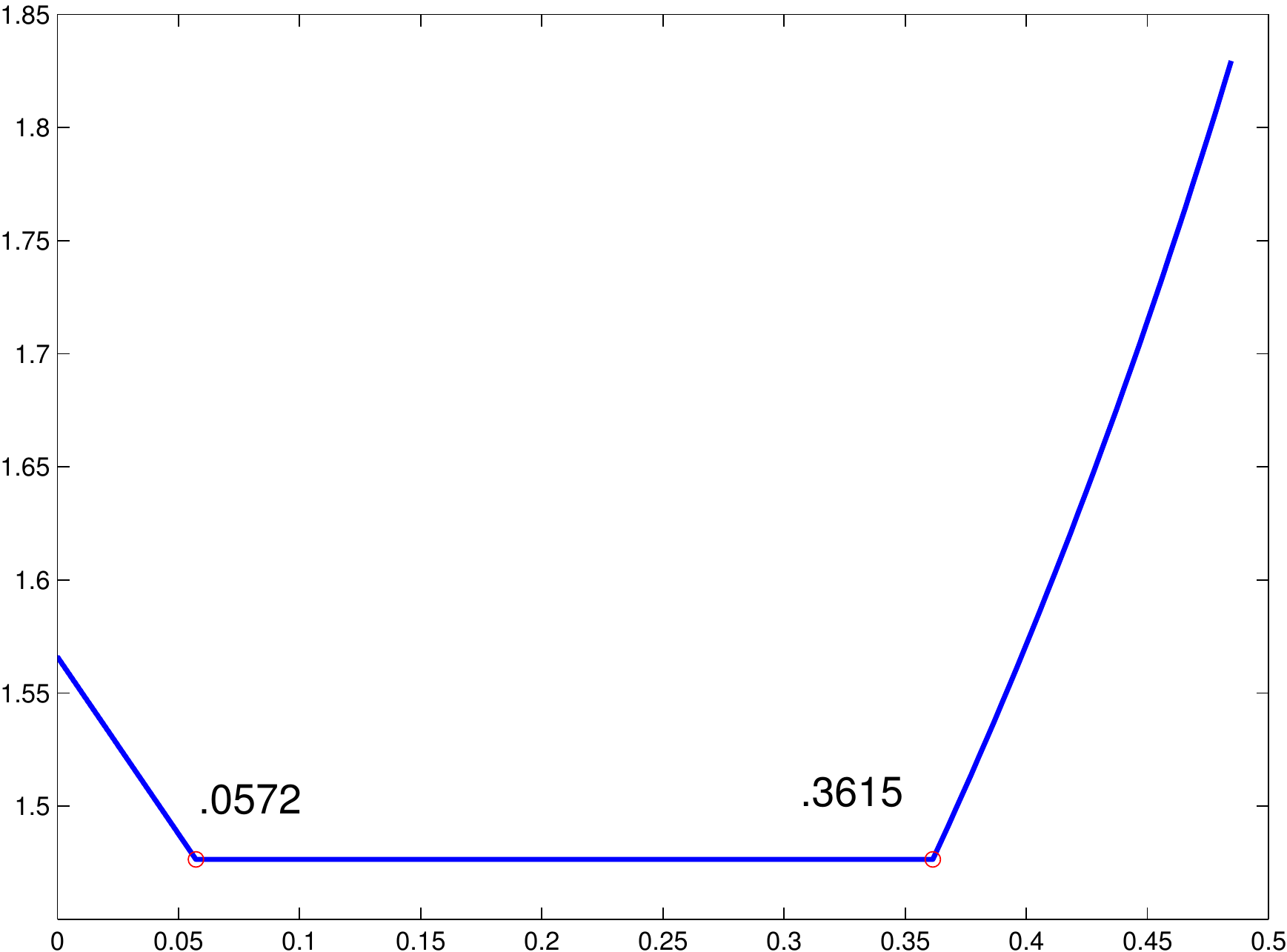}
\end{minipage}
\hspace{1em}
\begin{minipage}[b][][c]{0.45\columnwidth}
\centering
\caption{Illustration of the condition number $\kappa(\theta)$ for the
  case of a 2-D Laplacian  matrix with $n_x=n_y=256$ and the number of
  the  subdomains  $p=2$, where  $64$  eigenvectors  are used  $(i.e.,
  k=64)$. $\lambda_s=.05719$,  $\lambda_{k+1}=.36145$, and the optimal
  condition number is $\kappa=1.4765$.
\label{fig:kappa}}
\end{minipage}
\end{figure}
Figure \ref{fig:kappa},  shows the variation of the condition number
$\kappa(\theta)$ as a function of
  $\theta$, for a 2-D Laplacian  matrix.  One may
conclude  from this result  that there  is no  reason for  selecting a
particular  $\theta\in   \left[\lambda_s,  \lambda_{k+1}\right]$  over
another one as long as $\theta$ belongs to the middle interval, since
 the spectral condition number $\kappa(\theta)$ is the same. In fact, in
practice when approximate eigenpairs are used, that are computed, for
example, by the Lanczos procedure, the choice $\theta = \lambda_{k+1}$
often gives  better performance than $\theta=\lambda_s$ in
this context because for  the former choice, the perturbed eigenvalues
are less likely to  be close to zero. An example  can be found in
Figure \ref{fig:preceig}, which shows  that when using accurate enough
eigenpairs,  both choices  of $\theta$  will give  the  same condition
number  (which  is also  the  optimal  one),  whereas when  relatively
inaccurate  eigenpairs are  used,  setting $\theta=\lambda_{k+1}$  can
give   a    better   condition    number   than   that obtained from   setting
$\theta=\lambda_{s}$.    In   what  follows,   we   assume  that   the
approximation   scheme  \eqref{eq:main}   is  used   with   $\theta  =
\lambda_{k+1}$, and we will  denote by $S_{k,\theta}\inv$ the related
approximate inverse of $S$.

\begin{figure}[h!]
\center
\caption{Illustration of the eigenvalues of $S\tS\inv$ for the case of
  a 2-D  Laplacian matrix with  $n_x=n_y=128$, the number  of subdomains
  $p=2$ and  the rank $k=16$,  such that the optimal  spectral condition number
  $\kappa(\theta)=3.0464$,    for    $\lambda_s    \le   \theta    \le
  \lambda_{k+1}$.  The  two  top   figures  show  the  eigenvalues  of
  $S\tS\inv$ with the  Ritz values and vectors from  $80$ steps of the
  Lanczos   iterations, where  for   both  choices   $\theta=\lambda_s$  and
  $\theta=\lambda_{k+1}$,  $\kappa(\theta)=3.0464$.   The  bottom  two
  figures show  the eigenvalues of $S\tS\inv$ in the cases  with $32$ steps  of the
  Lanczos    iterations,   where    $\kappa(\lambda_s)=7.8940$   while
  $\kappa(\lambda_{k+1})=6.6062$.    \label{fig:preceig}}  
  \subfigure[$\theta=\lambda_s$, $80$ Lanczos steps ]{
  \includegraphics[width=0.47\textwidth]{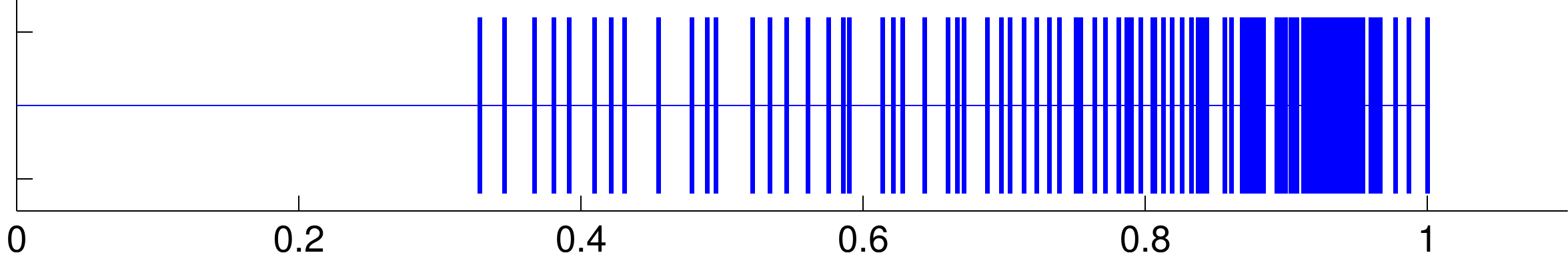}   }  
  \subfigure[$\theta=\lambda_{k+1}$, $80$ Lanczos steps]{
  \includegraphics[width=0.47\textwidth]{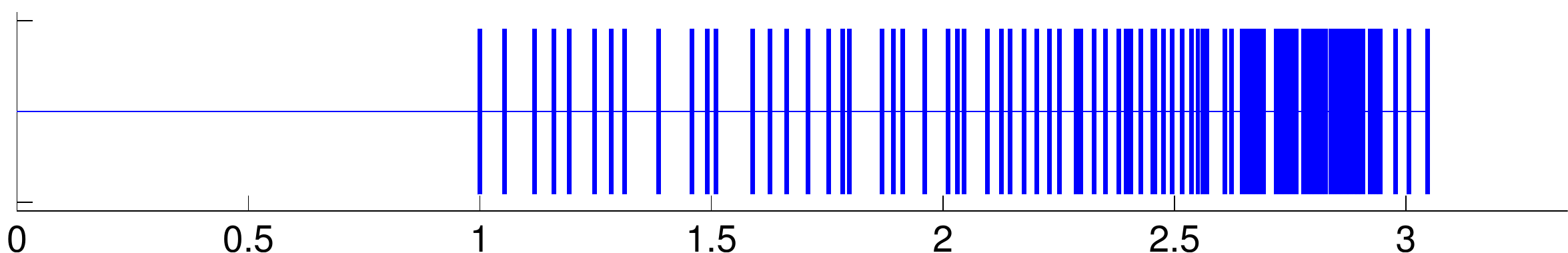}  }  
  \subfigure[$\theta=\lambda_s$, $32$ Lanczos steps]{
  \includegraphics[width=0.47\textwidth]{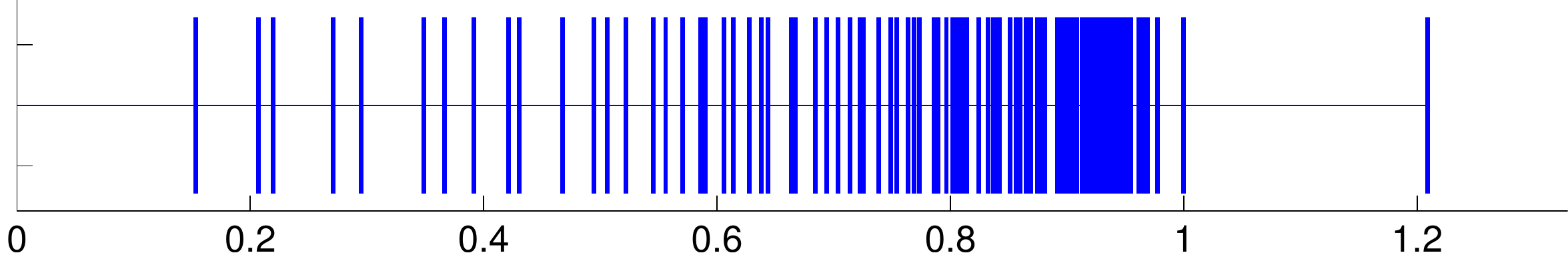}      }     
  \subfigure[$\theta=\lambda_{k+1}$, $32$ Lanczos steps]{
  \includegraphics[width=0.47\textwidth]{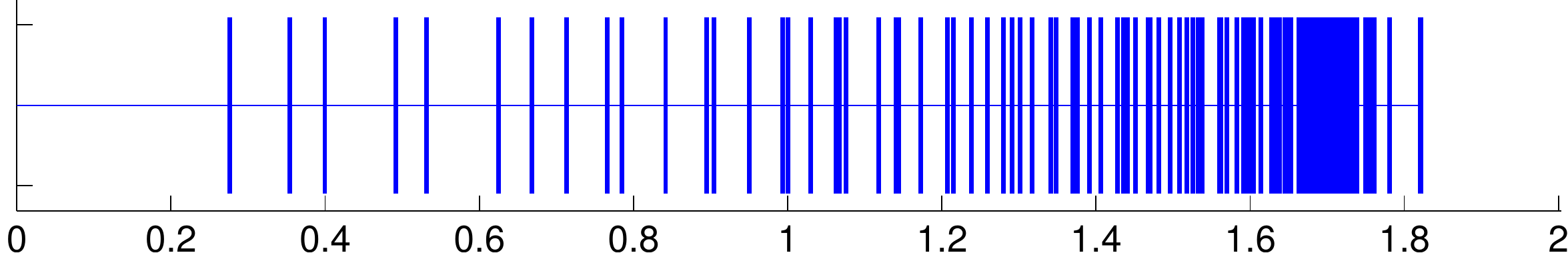} }
\end{figure}

From an implementation point of view, it is clear that only the $k$ largest eigenvalues and the associated eigenvectors as well as
the $(k+1)$-st largest eigenvalue of the matrix $C\inv E^TB\inv E$  are needed. We prove this result in the following proposition.

\begin{proposition}\label{prop:skt}
Let $Z_k$ be the eigenvectors of %the matrix
$C\inv E^TB\inv E$ associated with the $k$ largest eigenvalues, and let $\theta=\lambda_{k+1}$. The following
expression for $S_{k,\theta}\inv $ holds:
\begin{equation}\label{eq:skt}
S_{k,\theta}\inv
=
\frac{1}{1 - \theta} C\inv  \ + \
Z_k  \left[ (I- \Lambda_k)\inv - (1-\theta)\inv I \right] Z_k^T .
\end{equation}
\end{proposition}

\begin{proof}
We write $U = [U_k, W]$, where $U_k=\left[u_1,\ldots,u_k\right]$ contains the eigenvectors of $H$ associated with the largest $k$ eigenvalues and $W$ contains the
remaining columns $u_{k+1}, \cdots, u_s$. Note that $W$ is not available
but we use the fact that $W W^T = I-U_k U_k^T$ for the purpose of this
proof.
With this, \nref{eq:tildeSinv2} becomes:
\begin{align*}
S_{k,\theta}\inv &=
C\inv + L\invt[U_k, W] \xpmatrix{  (I- \Lambda_k)\inv -I  &  0 \\
0 & ((1-\theta)\inv -1) I } [U_k, W]^T L\inv   \\
&=
C\inv + Z_k \left[ (I- \Lambda_k)\inv -I \right] Z_k^T
+ \left[(1-\theta)\inv -1 \right] L\invt W W^T L\inv  \\
&=
C\inv + Z_k \left[ (I- \Lambda_k)\inv -I \right] Z_k^T
+ \left[(1-\theta)\inv -1 \right] L\invt(I- U_k U_k^T)L\inv  \\
&=
\frac{1}{1 - \theta} C\inv  + 
Z_k  \left[ (I- \Lambda_k)\inv - (1-\theta)\inv I \right] Z_k^T .
\end{align*}
\end{proof}

In a paper describing a similar technique, Grigori et al.
\cite{grigori:hal-01017448},  suggest 
another choice of  $\tilde{\Lambda}$ which is: 
\begin{equation} \label{eq:Lambda-2}
\tilde \lambda_i = \left\{ \begin{array}{cl}
1-(1-\lambda_i)/\varepsilon & \mbox{if} \quad  i\le k \\
 0  & \mbox{otherwise}
\end{array},
\right.
\end{equation}
where $\varepsilon$ is a parameter. Then the eigenvalues $\sigma_i$'s are
\[
\left\{ \begin{array}{cl}
\varepsilon   & \mbox{if} \quad  i\le k \\
1-\lambda_i  & \mbox{otherwise} 
\end{array},
\right.
\]
Note that the first choice in \eqref{eq:Lambda-1} is a special case of
\eqref{eq:Lambda-2} when $\varepsilon=1$.
Writing the transformed eigenvalues as
 \[ \{\eps, 1-\lambda_{k+1},1-\lambda_{k+2},\cdots, 1-\lambda_{s} \} , \]
the authors stated that 
the resulting condition number is $\kappa=(1-\lambda_s)/\eps$, with 
an implied assumption that $\eps \le 1-\lambda_{k+1}$. 
In the cases when $ 1-\lambda_{k+1} < \eps \le 1 -\lambda_s$, 
the spectral condition number is  the same as above,
i.e., equal to $(1-\lambda_s)/(1-\lambda_{k+1}) $. 
On the other hand, when 
$ 0 \le \eps \le 1-\lambda_{k+1}$, then the condition number is now 
$(1-\lambda_s)/\eps$, and the best value will be reached again for 
$\eps = 1-\lambda_{k+1}$, which  leads to the same condition number as above.
\begin{comment}
 Regardless of this optimal condition number,
notice that the  assumption $\eps \le 1-\lambda_{k+1}$ implies
that $\lambda_{k+1} \le  1-\eps$.  If we call $K$ the condition number as in the
previous discussion, then $K=1/\eps$ and the condition 
$\lambda_{k+1} \le  1-\eps$ means that 
$\lambda_{k+1} \le  1-1/K$ which is identical with 
\nref{eq:condk}.
\end{comment}

In all the cases, if we want to keep the spectral condition number of the matrix $S\tS\inv$, which is $\kappa=(1-\lambda_s)/(1-\lambda_{k+1})$, bounded from above by a constant $K$, 
we can only guarantee this by having $k$ large enough so that 
$ 1/(1-\lambda_{k+1}) \le K $, or equivalently, $\lambda_{k+1} \le 1 - 1/K$.
In other words, we would have to select the rank $k$ large enough such that 
\eq{eq:condk} 
\lambda_{k+1} \le 1 - \frac{1}{K} \ .
\en
%For example if $\kappa=10$, then we need to use a $k$ large enough 
%to capture all eigenvalues of $S$ that are larger than 0.9.
Of course, the required  rank $k$  depends primarily on  the eigenvalue
decay of  the $\lambda_i$'s. In general, however, this means  that the 
method  will require  a sufficient number  of eigenvectors  to be
computed and  that this number must  be increased if  we wish to decrease
the spectral condition number to a given value. For problems 
arising from PDEs,
it is  expected that in order  to keep the spectral condition number constant,
$k$ must have to be increased as the problem sizes increase.

\begin{comment}
Here is now a result that shows that the best that can be done from a 
low-rank approximation is a condition number that equals 
to $(1-\lambda_s)/(1-\lambda_{k+1})$.

\begin{theorem} 
Let $H$ be approximated as in \nref{eq:Happrox} where $\tilde \Lambda$
is any $s\times s$ symmetric matrix of rank $k$. Then the 
spectral condition number of $S \tilde S\inv$ is greater than or equal to
$(1-\lambda_s)/(1-\lambda_{k+1})$ and the minimum is reached when $\tilde \Lambda = \Lambda_k$.
\end{theorem}

\begin{proof}
From \eqref{eq:tildeSinv2}, we can write
\[
\tilde{S}=C\inv + L\invt U \hat{\Lambda} U^TL\inv, \quad \mathrm{with} \quad \hat{\Lambda}=(I-\tilde{\Lambda})\inv-I,
\]
where $\hat{\Lambda}$ is also of rank $k$.
The preconditioned matrix reads 
\begin{align}
S\tilde{S}\inv &= (C-E^TB\inv E)(C\inv+L\invt U \hat{\Lambda} U^TL\inv) \\
&= I - E^TB\inv E C\inv + SL\invt U \hat{\Lambda} U^TL\inv
\end{align}
Recall that 
\begin{align*}
E^TB\inv E C\inv &= LU\Lambda U^T L\inv, \\
SL\invt U &= LU(I-\Lambda),
\end{align*}
and we have
\begin{align}
S\tilde{S}\inv &=LU(I-\Lambda)U^TL\inv + LU(I-\Lambda) \hat{\Lambda} U^TL\inv \\
&=LU(I-\Lambda)(I-\tilde{\Lambda})\inv U^TL\inv.
\end{align}
Here we can see that 
\end{proof}
\end{comment}

\section{Practical implementation} \label{sec:impl}
In this section, we will address the implementation details for building and applying an SLR preconditioner.

\subsection{Computation of the low-rank approximations}
One of the key issues in setting up the preconditioner \eqref{eq:M} is
to extract a low-rank approximation  to the matrix $C\inv E^TB\inv E$.
Assuming  that  $C$   is  SPD,  we  can  use   the  Lanczos  algorithm
\cite{GVL-Book,lanczos} on the  matrix $L\inv E^TB\inv E L\invt$,  where $L$ is
the  Cholesky factor  of $C$.  In  the case  when only  a few  extreme
eigenpairs   are  needed,  the   Lanczos  algorithm   can  efficiently
approximate  these without  forming  the matrix  explicitly since  the
procedure  only requires the  matrix for  performing the matrix-vector
products.
As is well-known, in the presence of rounding errors, 
orthogonality in the Lanczos procedure
is quickly lost and a form of reorthogonalization is needed in practice.
In our approach, the partial
reorthogonalization scheme \cite{parlett1979,simon1984} is used.
The cost of this step will not be an issue to the overall performance
when a small number of steps are performed to approximate a few
eigenpairs.
\subsection{The solves with $\mathrm{B}$ and $\mathrm{C}$}
 A   solve with  the matrix  $B$ amounts to $p$ local and independent solves with the matrices
$B_i$, $i=1,\cdots,p$.
These can be carried out efficiently either by a direct
solver or by Krylov subspace methods with  more traditional ILU
preconditioners for example.
\begin{comment}
\begin{figure}[h!t]
\centering
\caption{The structure of the multilevel DD method with vertex separators. The proposed SLR method can be recursively applied on the $C$ matrix (the block in the dashed red box). \label{fig:pattern}}
\vspace*{.5em}
\includegraphics[scale=0.35]{FIGS/pattern}
\end{figure}
\end{comment}
On  the  other hand,  the  matrix $C$,  which  is  associated with  the
interface unknowns,  often has  some diagonal dominance  properties for
problems issued from discretized PDEs, so that an ILU-based method can
typically  work   well.   However,  for   large  indefinite  problems,
especially ones issued from 3-D PDEs,
the interface corresponds to a
large 2-D problem, and so 
a direct  factorization of $C$ will be expensive  
in terms of both the memory and the computational cost.
An  alternative is to  apply the  SLR
method recursively. This requires that the interface  points be 
ordered so  that $C$ will have the  same structure as the 
matrix $A$. That is
the leading block is block diagonal, a property satisfied by
the Hierarchical Interface Decomposition (HID)  method 
discussed in \cite{HenonSaad}.
This essentially yields a multilevel scheme of the SLR method,
 which is currently being investigated by the authors.
In the current SLR method, we simply use ILU factorizations for  $C$.

\begin{comment}

Delete? Build the approximation gradually using ``self-preconditioning''.
This entails starting the solves with $S$ then obtaining a low-rank
approximation which can be used to precondition $S$ for the next 
iteration etc.. 
\end{comment}

\subsection{Improving an SLR preconditioner} \label{sec:improv}
One of the main weaknesses of standard, e.g., ILU-type, preconditioners
is that they are difficult to update. For example, suppose we 
compute a preconditioner  to a given matrix  and find
that it is not accurate enough to yield convergence. In the 
case of ILUs we would have essentially to start from the beginning.
However, for SLR, improving a given preconditioner is essentially trivial.
For example, the heart of the SLR method consists of obtaining a low-rank
approximation the  matrix $H$ defined in \nref{eq:H}. 
Improving this approximation would consist in merely adding a few more
vectors (increasing $k$) and this can be easily achieved in a number of
ways, e.g., by resorting to a form of deflation, without having to throw away the vectors already computed.

\section{Numerical experiments} \label{sec:exp}

The experiments were conducted on a machine
at  Minnesota Supercomputing Institute,
equipped with two
Intel Xeon X5560 processors ($8$ MB Cache, $2.8$ GHz, quad-core)
and $24$ GB of main memory.
A preliminary implementation of the SLR preconditioner was
written in C/C++, and the code was compiled by the Intel C compiler using the
-O2 optimization level. BLAS and LAPACK routines from Intel Math Kernel Library
were used to enhance the performance on multiple cores.
The thread-level parallelism was realized by OpenMP \cite{openmp11}.

The accelerators used were the conjugate gradient (CG)
method for the SPD cases, and the generalized
minimal residual (GMRES) method with a restart dimension of $40$, denoted by
GMRES$(40)$ for the indefinite cases.
Three types of preconditioning methods were compared in our experiments:
 the incomplete Cholesky factorization with threshold dropping (ICT) or the
incomplete LDL factorization with threshold dropping (ILDLT),
the restricted additive Schwarz (RAS) method~\cite{RAS} (with one-level
overlapping), and the SLR method.
For the RAS method, we used ICT/ILDLT as the local solvers.
Moreover, since the RAS preconditioner is nonsymmetric even for a symmetric
matrix, GMRES$(40)$ was used with it.

For all the problems, we used the graph partitioner
\texttt{PartGraphRecursive} from \textsc{Metis} \cite{METIS-SIAM,Karypis199871}
to partition the domains.
The time for the graph partitioning will not be included in the time of building
the preconditioners.
%%%, which is, however, considered as a preprocessing step.
For each subdomain $i$, matrix $B_{i}$ was reordered by the
approximate minimum degree  (AMD) ordering \cite{AmestoyDavisDuff96,
AmestoyDavisDuff04,Davis:2006:DMS:1196434} to reduce fill-ins
and then ICT/ILDLT was used as a local solver.
In the SLR method, the matrix $C$, which is assumed to be SPD, was factored by
ICT.
In the Lanczos algorithm, we set the maximum number
of Lanczos steps as five times the number of requested eigenvalues.

Based on the experimental results, we can state
that in general, building an SLR preconditioner, especially those using larger
ranks, requires much more time
than an ICT/ILDLT preconditioner or an RAS preconditioner that requires similar
storage. Nevertheless, experimental results indicated that the SLR
preconditioner is more robust and can achieve great time savings in the
iterative phase.
%In particular, in the cases of systems with a large number of right-hand sides,
expensive but effective preconditioners may be justified because their cost is
amortized.
In this section, we first report on the results of solving symmetric linear
systems from a 2-D/3-D  PDE on regular meshes. Then, we will show the results
for solving a sequence of
general sparse symmetric linear systems.
For all the cases, the iterations were stopped whenever the residual norm had
been reduced by $8$ orders of magnitude or the maximum number of iterations
allowed,
which is $300$, was exceeded.
The results are summarized in Tables \ref{tab:spd2d3d}, \ref{tab:indef2d3d} and
\ref{tab:genmatresults}, where
all times are reported in seconds.
When comparing the preconditioners, the following factors are considered:
1) fill-ratio, i.e., the ratio of the number of nonzeros required to store
    a preconditioner to the number of nonzeros in the original matrix,
2) time for building preconditioners,
3) the number of iterations and
4) time for the iterations.
In all tables, `\texttt{F}' indicates non-convergence within the maximum allowed
number of steps.

\subsection{2-D  and 3-D model problems}
We  examine problems from a 2-D/3-D  PDE,
\begin{align} \label{eq:2d3d-pde}
-\Delta u-cu &= f\:\textrm{ in }\Omega, \notag \\
u &= \phi(x) \textrm{ on } \partial \Omega,
\end{align}
where $\Omega = \left(0,1\right)^2$ and $\Omega =
\left(0,1\right)^3$ are the  domains,
and $\partial \Omega$ is the boundary.
We take the $5$-point or $7$-point centered difference approximation on the
regular meshes.

To begin  with, we  examine the  required ranks of  the SLR  method in
order to bound the spectral  condition number of the matrix $S\tS\inv$
by a constant $K$. Recall from \eqref{eq:condk} that this requires that the
$(k+1)$-st largest  eigenvalue, $\lambda_{k+1}$, of  the matrix $C\inv
E^T B\inv E$ be less than $1-1/K$.  The results for 2-D/3-D Laplacians
are shown in Table \ref{tab:rank}. From  there we can see that for the
2-D problems, the required rank is about doubled when the step-size is
reduced  by half,  while  for the  3-D  cases, the  rank  needs to  be
increased by a factor of roughly $3.5$.

\begin{table}[ht]
\caption{The required ranks of the SLR method for bounding the condition number
of $S\tS\inv$ by $K$ for 2-D/3-D Laplacians. The number of subdomains used is
$8$ for the 2-D case and $32$ for the 3-D case. \label{tab:rank}}
\begin{center}
\begin{tabular}{c|c||c|c}
\hline
\multirow{2}{*}{Grid} & rank & \multirow{2}{*}{Grid} & rank\tabularnewline
 & $K\approx33$ &  & $K\approx12$\tabularnewline
\hline
$128^{2}$ & 3 & $25^{3}$ & 1\tabularnewline

$256{}^{2}$ & 8 & $40^{3}$ & 4\tabularnewline

$512{}^{2}$ & 20 & $64^{3}$ & 12\tabularnewline

$1024^{2}$ & 42 & $100^{3}$ & 42\tabularnewline
\hline
\end{tabular}
\end{center}
\end{table}

In  the next  set of  experiments, we  solve  \eqref{eq:2d3d-pde} with
$c=0$, so  that the coefficient  matrices are SPD  and we use  the SLR
preconditioner along  with the CG method.   Numerical experiments were
carried out to compare the  performance of the SLR preconditioner with
the ICT  and the RAS preconditioners.  The results are  shown in Table
\ref{tab:spd2d3d}. The sizes of the grids, the fill-ratios (fill), the
numbers of iterations (its), the time for building the preconditioners
(p-t) and  the time for iterations  (i-t) are tabulated.   For the SLR
preconditioners, the number  of subdomains (nd) and the  rank (rk) are
also  listed.   The  fill-ratios  of the  three  preconditioners  were
controlled to be roughly equal.  For  all the cases tested here and in
the following sections, the RAS method always used the same numbers of
subdomains as did the SLR method.  The ICT factorizations  were used
for the  solves with the matrices $B$  and $C$ in the  SLR method.  As
shown in Table~\ref{tab:spd2d3d}, we  tested the problems on three 2-D
grids  and three  3-D grids  of increasing  sizes, where  for  the RAS
method and the SLR method, the domain was partitioned into $32$, $64$
and $128$  subdomains respectively,  and the ranks  $16$ or  $32$ were
used in the SLR preconditioners.
\begin{table}[ht!]
  \caption{Comparison among the {\rm ICT}, the {\rm RAS} and the {\rm SLR}
preconditioners for solving  SPD linear systems from the 2-D/3-D PDE in
\eqref{eq:2d3d-pde} with $c=0$ along with the {\rm CG} and the {\rm GMRES}
method. \label{tab:spd2d3d}}
  \centering\tabcolsep1.5mm
  {%\fontsize{8.5}{1.1em}\selectfont
  \renewcommand{\arraystretch}{1.15}
  \begin{center}
  \begin{tabular}{r|rrrr|rrrr|rrrrrr}
    \hline %\noalign{\smallskip}
    \multicolumn{1}{c|}{\multirow{2}{*}{Grid}} &
    \multicolumn{4}{c|}{ICT-CG} &
    \multicolumn{4}{c|}{RAS-GMRES} &
    \multicolumn{6}{c}{SLR-CG} \tabularnewline
    & fill & p-t & its & i-t
    & fill & p-t & its & i-t
    & nd & rk & fill & p-t & its & i-t \tabularnewline \hline
    $256^2$  & 4.5 & .074 & 51  & .239 & 4.5 & .088 & 129 & .281 & 32 & 16 & 4.3
& .090 & 67  & .145 \tabularnewline
    $512^2$  & 4.6 & .299 & 97  & 1.93 & 4.8 & .356 & 259 & 2.34 & 64 & 32 & 4.9
& .650 & 103 & 1.01 \tabularnewline
    $1024^2$ & 5.4 & 1.44 & 149 & 14.2 & 6.2 & 1.94 & \tF & 12.8 & 128 & 32 &
5.7 & 5.23 & 175 & 7.95 \tabularnewline
    \hline
    $40^3$   & 4.4 & .125 & 25  & .152 & 4.5 & .145 & 36 & .101 & 32 & 16 & 4.0
& .182 & 31  & .104 \tabularnewline
    $64^3$   & 6.8 & .976 & 32 & 1.24 & 6.2 & .912 & 49 & .622 & 64 & 32 & 6.3 &
1.52 & 38 & .633 \tabularnewline
    $100^3$  & 7.3 & 4.05 & 47 & 7.52 & 6.1 & 3.48 & 82 & 4.29 & 128 & 32 & 6.5
& 5.50 & 67 & 4.48 \tabularnewline
    %\noalign{\smallskip}
    \hline
  \end{tabular}
  \end{center}
  }
\end{table}

Compared  with the  ICT and  the  RAS preconditioners,  building an  SLR
preconditioner required  more CPU time (up  to $4$ times  more for the
largest  2-D case).  For  these problems,  the SLR-CG  method achieved
convergence  in  slightly more  iterations  than  those  with the  ICT
preconditioner, but  SLR  still achieved performance gains
in terms of significantly reduced iteration times.  The
CPU time for building an  SLR preconditioner is typically dominated by
the cost of the Lanczos algorithm.  Furthermore, this cost is actually
governed by  the cost of the  solves with  $B_i$'s and $C$,
which are  required at  each iteration.  Moreover,  when the  rank $k$
used  is  large, the  cost  of  reorthogonalization  will also  become
significant.  Some simple thread-level parallelism has been exploited
using OpenMP for the solves with the $B_i$'s,  which can be performed
independently. The  multi-threaded MKL routines also  helped speedup the
vector  operations in  the  reorthogonalizations.  We  point out  that
there is room for substantial improvements in the performance of these
computations. In particular  they are very suitable for  the SIMD type
parallel machines  such as  computers equipped with  GPUs or  with the
Intel  Xeon  Phi  processors.    These  features  have  not  yet  been
implemented in the current code.

Next, we consider solving the symmetric indefinite problems by setting $c>0$ in
\eqref{eq:2d3d-pde}, which corresponds to shifting the discretized negative
Laplacian (a positive definite matrix) by subtracting  $s I$ with a certain
$s>0$.
In this set of experiments, we solve the 2-D problems with $s=0.01$ and the 3-D
problems with $s=0.05$. The SLR method is compared to ILDLT and  RAS   with
GMRES$(40)$.
\begin{table}[ht]
\caption{Comparison among the {\rm ILDLT}, the {\rm RAS} and the {\rm
SLR} preconditioners for solving symmetric indefinite linear systems from the
2-D/3-D PDE in \nref{eq:2d3d-pde} with $c>0$ along with the {\rm
GMRES} method. \label{tab:indef2d3d}}
  \tabcolsep1.4mm
  \begin{center}
  {%\fontsize{8.5}{1.1em}\selectfont
  \renewcommand{\arraystretch}{1.15}
  \begin{tabular}{r|rrrr|rrrr|rrrrrr}
    \hline
    \multicolumn{1}{c|}{\multirow{2}{*}{Grid}} &
    \multicolumn{4}{c|}{ILDLT-GMRES} &
    \multicolumn{4}{c|}{RAS-GMRES} &
    \multicolumn{6}{c}{SLR-GMRES} \tabularnewline
    & fill & p-t & its & i-t & fill & p-t & its & i-t
    & nd & rk & fill & p-t & its & i-t \tabularnewline \hline
    $256^2$  & 8.2 & .174 & \tF &  --  & 6.3 & .134 & \tF &  --  & 8 & 32 & 6.4
& .213 & 33 & .125 \tabularnewline
    $512^2$  & 8.4 & .702 & \tF &  --  & 8.4 & .721 & \tF &  --  & 16 & 64 & 7.6
& 2.06 & 93 & 1.50 \tabularnewline
    $1024^2$ & 12.6 & 5.14 & \tF &  --  & 19.4& 21.6 & \tF &  --  & 8 & 128 &
10.8 & 24.5 & 50 & 4.81 \tabularnewline
    \hline
    $40^3$   & 6.9  & .249 & 54  & .540 & 6.7 & .254 & 99 & .300 & 64 & 32 & 6.7
& .490 & 23 & .123 \tabularnewline
    $64^3$   & 9.0  & 1.39 & \tF &  --  & 11.8& 2.16 & \tF &  --  & 128 & 64  &
9.1 & 3.94 & 45 & 1.16 \tabularnewline
    $100^3$  & 14.7 & 10.9 & \tF &  --  & 11.7& 14.5 & \tF &  --  &128 & 180 &
14.6 & 62.9 & 88 & 13.9 \tabularnewline
    \hline
  \end{tabular}}
  \end{center}
\end{table}

Results  are shown in  Table~\ref{tab:indef2d3d}.  For  most problems,
the  ILDLT/GMRES  and  the  RAS/GMRES  method failed  even  with  high
fill-ratios. In contrast, the SLR method appears to be more effective,
achieving  convergence for all cases, and great savings in the
iteration time. In contrast with  the SPD case,
a few  difficulties were
encountered.  For the 2-D  problems,  an SLR preconditioner with
a large  number of subdomains  (say, $64$ or  $128$)  often failed
to converge. As a result  the sizes of the subdomains were still quite
large and factoring the matrices $B_i$'s  was quite  expensive
in terms of both the CPU time and the memory requirement. Furthermore,
for both  the 2-D  and   3-D  problems, approximations of  higher ranks
were required compared to those  used in the SPD cases. This only
increased the memory requirement slightly,  but it significantly increased the
CPU
time required by the Lanczos  algorithm. An example is the largest 3-D
problem in Table~\ref{tab:indef2d3d}, where a rank of $180$ was used.

\subsection{General matrices}
We selected $15$ matrices from the University of Florida sparse matrix
collection \cite{Davis:UFM} for the following tests.
Among these $10$ matrices are SPD matrices and
$5$ matrices are symmetric indefinite.
Table \ref{tab:matrices} lists the name, the order (N),
the number of nonzeros (NNZ), the positive definiteness,
and a short description for each matrix.
 If the actual right-hand side is not provided,
the linear system was  obtained  by creating an
 artificial one as $b=Ae$, where $e$ is a random vector of unit $2$-norm.

\begin{table}[ht]
\caption{Names, orders (N), numbers of nonzeros (NNZ)
and positive definiteness  of the test matrices.}
\begin{center}\tabcolsep1.7mm
{%\fontsize{9.5}{1.1em}\selectfont
\renewcommand{\arraystretch}{1.05}
\begin{tabular}{lrrrl}
\hline\noalign{\smallskip}
\multicolumn{1}{c}{MATRIX} & \multicolumn{1}{c}{N} & \multicolumn{1}{c}{NNZ} &
\multicolumn{1}{c}{SPD} & \multicolumn{1}{c}{DESCRIPTION}\tabularnewline
\noalign{\smallskip}\hline\noalign{\smallskip}
Williams/cant    & 62,451    & 4,007,383 & yes & FEM cantilever  \tabularnewline
UTEP/dubcova2    & 65,025    & 1,030,225 & yes & 2-D/3-D PDE problem
\tabularnewline
UTEP/dubcova3    & 146,689    & 3,636,643 & yes & 2-D/3-D PDE problem
\tabularnewline
Rothberg/cfd1    & 70,656    & 1,825,580 & yes & CFD problem \tabularnewline
Rothberg/cfd2    & 123,440   & 3,085,406 & yes & CFD problem \tabularnewline
Schmid/thermal1  & 82,654    & 574,458   & yes & thermal problem \tabularnewline
Schmid/thermal2  & 1,228,045 & 8,580,313 & yes & thermal problem \tabularnewline
Wissgott/parabolic\_fem & 525,825 & 3,674,625 & yes & CFD problem
\tabularnewline
CEMW/tmt\_sym & 726,713 & 5,080,961 & yes & electromagnetics problem
\tabularnewline
McRae/ecology2 & 999,999 & 4,995,991 & yes & landscape ecology problem
\tabularnewline
\hline
Lin/Lin & 256,000 & 1,766,400 & no & structural problem  \tabularnewline
Cote/vibrobox    & 12,328    & 301,700   & no  & vibroacoustic problem
\tabularnewline
Cunningham/qa8fk & 66,127    & 1,660,579 & no  & 3-D acoustics problem
\tabularnewline
Koutsovasilis/F2 & 71,505    & 5,294,285 & no  & structural problem
\tabularnewline
GHS\_indef/helm2d03 & 392,257 & 2,741,935 & no & 2-D Helmholtz problem
\tabularnewline
%\noalign{\smallskip}
\hline
\end{tabular}
}
\end{center}
\label{tab:matrices}
\end{table}

Table~\ref{tab:genmatresults} shows the performance of the three
preconditioning methods.
The CG method and the GMRES method with the SLR
preconditioner achieved convergence for all the cases, whereas for many cases,
they failed to converge with the ICT/ILDLT and the RAS preconditioners.
Similar to the experiments for the model problems, the SLR
preconditioner often required more CPU time to build than
the other two counterparts but it required fewer iterations for most of the
cases and achieved significant CPU time savings in the iteration phase for
almost all the cases (the exception is {\tt qa8fk}, for which the RAS method
gave the best iteration time).

\begin{table}[ht]
\caption{
  Comparison among the {\rm ICT} or the {\rm ILDLT}, the {\rm RAS} and the {\rm
SLR}
  preconditioners for solving general symmetric linear systems
  along with the {\rm CG} or {\rm GMRES$(40)$} method.
\label{tab:genmatresults}}
  \tabcolsep1.325mm
  \begin{center}
  {%\fontsize{8}{1.1em}\selectfont
  \renewcommand{\arraystretch}{1.1}
\begin{tabular}{l|rrrr|rrrr|rrrrrr}
    \hline %\noalign{\smallskip}
    \multicolumn{1}{c|}{\multirow{2}{*}{MATRIX}} &
    \multicolumn{4}{c|}{ICT/ILDLT} &
    \multicolumn{4}{c|}{RAS} &
    \multicolumn{6}{c}{SLR} \tabularnewline
    & fill & p-t & its & i-t
    & fill & p-t & its & i-t
    & nd & rk & fill & p-t & its & i-t \tabularnewline \hline
cant & 4.7  &  3.87  & 150  & 9.34  & 5.9  & 6.25  & \tF  &  --  & 32 & 90 & 4.9
  & 5.58  & 82  & 1.92 \tabularnewline
dubcova2 & 2.7  &  .300  & 47   & .492   & 2.8   & .489  & 60    & .223  & 16 &
32 & 2.8   & .280  & 19    & .080 \tabularnewline
dubcova3 & 2.2  &  1.01  & 46   & 1.44   & 2.1   & 1.46  & 59    & .654  & 16 &
32 & 1.8   & .677  & 19    & .212 \tabularnewline
cfd1 & 6.9  &  2.89  & 295  & 11.9  & 8.3   & 3.04  & \tF   &  --    & 32 & 32 &
6.9   & 2.13  & 64    & 1.07 \tabularnewline
cfd2 & 9.9  & 13.5  & \tF  &  --   & 8.9   & 7.88  & \tF   &  --     & 32 & 80 &
8.8   & 7.62  & 178   & 5.75 \tabularnewline
thermal1 & 5.1  &  .227  & 68   & .711   & 5.0   & .348  & \tF   &  --   & 16 &
32 & 5.0   & .277  & 59    & .231 \tabularnewline
thermal2 & 6.9  &  5.10  & 178  & 39.3  & 7.1   & 8.46  & \tF   &  --   & 64 &
90 & 6.6   & 14.8 & 184   & 15.0 \tabularnewline
para\_fem & 6.1  &  2.04  & 58   & 4.68   & 6.3   & 3.17  & 236   & 6.11 & 32 &
80 & 6.9   & 6.05  & 86    & 3.03 \tabularnewline
tmt\_sym & 6.0  &  1.85  & 122  & 11.6  & 6.2   & 3.67  & \tF   &  --   & 64 &
80 & 5.9   & 6.61  & 127   & 5.23 \tabularnewline
ecology2 & 8.4  &  2.64  & 142  & 18.5  & 9.5   & 4.78  & \tF   &  --   & 32 &
96 & 8.0   & 12.3 & 90    & 5.58 \tabularnewline
\hline
Lin & 11  &  1.93  & \tF  &  --   & 19  & 4.61  & \tF   &  --  & 64  & 64 & 9.9
 & 3.78  & 73    & 1.75 \tabularnewline
vibrobox & 6.0  &  .738  & \tF  &  --    & 7.0   & .513  & \tF   &  --   & 4 &
64 & 3.8   & .437  & 226   & .619 \tabularnewline
qa8fk & 4.2 & .789 &  22  &  .507 & 4.6 & 1.14 &  35  & .273 & 16 & 64 & 4.5 &
1.94 & 28 & .309 \tabularnewline
F2 & 5.1 & 9.66 &  \tF &   --  & 5.4 & 9.43 &  \tF &  --  & 8 & 80 & 3.9 & 6.25
& 72 & 2.14 \tabularnewline
helm2d03 & 14 & 14.4 &  \tF &   --  & 11 & 7.20 &  \tF &  --  & 16 & 128 & 11 &
11.9 & 63 & 2.63 \tabularnewline
\hline
\end{tabular}
}
  \end{center}
\end{table}

\section{Conclusion} \label{sec:concl}
This paper presented a  preconditioning method, named SLR, 
based  on a Schur  complement
approach with  low-rank corrections for solving
symmetric  sparse  linear  systems. 
Like the method in \cite{MLR-1}, the new  method 
uses  a low-rank approximation
to build a preconditioner, exploiting some decay property of 
eigenvalues. The major difference with 
 \cite{MLR-1} is that SLR is not recursive. It focuses on the
Schur complement in any standard domain decomposition framework
and tries to approximate its inverse by exploiting low-rank
approximations. As a result, the method is much easier
to implement.

Experimental  results indicate
that in terms of iteration times, the  proposed preconditioner can  be a more
efficient alternative to the  ones based on incomplete factorizations,
namely,  the  ILU-type or  block  ILU-type  methods 
 for     SPD systems.  Moreover,  this  preconditioner appears  to be  more
robust than the incomplete  factorization based methods for indefinite
problems.  Recall  that ILU-based methods often  deliver unstable, and
in  some  cases  quite  dense  factors when  the  original  matrix  is
highly indefinite, and this  renders  them  ineffective  for  such  cases.  
In contrast  SLR is essentially a form of approximate inverse technique
and as such it is not prone to these difficulties.
On  the negative  side, building  an  SLR preconditioner  can be  time
consuming, although  several mitigating  factors should be  taken into
account. These are similar to those pointed out in \cite{MLR-1} which 
also exploits low-rank approximation and we summarize them here.
The first is  that a big part of the computations to build the 
SLR  preconditioner  can be  easily  vectorized and this is  especially
attractive for massively  parallel machines,  
such as those equipped with GPUs  or 
with the Intel Xeon Phi processors.  The set-up phase is likely to be
far more advantageous than a factorization-based one which tends to be
much more  sequential, see, e.g., \cite{RliSaadGPU}.   
The second is that 
there are situations in which many systems with the same matrix
must  be  solved in which  case  more  expensive but  more  effective
preconditioners may be justified as their cost will  be amortized.
Finally, these preconditioners 
are more easily updatable than traditional ILU-type preconditioners, see
Section~\ref{sec:improv} for a discussion.

\section*{Appendix} 
Let
\[
T = 
\begin{pmatrix} 
 2a &   -1    &           &                   \cr
-1  &    2a   &  -1       &                   \cr
    & \ddots  &  \ddots   & \ddots            \cr
    &         &  \ddots   & \ddots  & -1      \cr
    &         &           &  -1     &   2a         
\end{pmatrix}
\]
and
\[
T=
\begin{pmatrix} 
  \hspace{10pt}  1    &         &     &                   \cr
-d_1 \inv     &   1     &           &                   \cr
              & -d_2\inv   &  \ddots   &                   \cr
    &         &  \ddots   & \ddots  &         \cr
    &         &           &  - d_{n}\inv     &  1           
\end{pmatrix} 
\begin{pmatrix} 
 d_1  &  - 1       &           &                   \cr
       &   d_2     &  -1       &                   \cr
    &         &  \ddots   & \ddots            \cr
    &         &           & \ddots  &  -1     \cr
    &         &           &     &   d_n     
\end{pmatrix} 
\]
be the LU factorization of $T$.
We are interested in $d_n\inv $. If we solve $T x = e_n $ where $e_n$ is the $n$th canonical basis vector for $\RR^{n}$, and
$x = [\xi_0, \cdots, \xi_{n-1}]^T$, then clearly $\xi_{n -1} =  1/d_n$
which is what we need to calculate. 
Let $\xi_k = U_k(a) $, for $k=0,1,\cdots,n-1$, where $U_k$ is the $k$-th degree Chebyshev 
polynomial of the second kind. These polynomials satisfy the 
recurrence relation: $U_{k+1} (t) = 2 t U_k (t) - U_{k-1}(t) $,
starting with $U_0 (t) = 1$ and $U_1(t) = 2 t$.
Then clearly, equations  $k=1,\cdots, n-1$ of the system $Tx = e_n$ 
 are satisfied. For the last equation we get $U_n(a)$ instead of 
the wanted value of 1.  Scaling
$x$ by $U_n(a) $ yields the result $1/d_n  = \xi_{n-1} =  U_{n-1}(a)/U_{n} (a) $.

\section*{Acknowledgements}
The authors  are grateful to 
the University of Minnesota
Supercomputing Institute
for providing them with computational resources and
assistance with the computations.

\bibliographystyle{siam}
\bibliography{strings,local}
\end{document}